\documentclass[english,a4paper]{amsart}

\usepackage{latexsym,amsthm,amscd}

\usepackage[all]{xy}
\usepackage{rotating}
\usepackage[utf8]{inputenc}
\usepackage[T1]{fontenc}
\usepackage{babel}
\usepackage{amsmath,leftidx}
\usepackage{accents}
\usepackage{amssymb}
\usepackage{amsfonts}
\usepackage{enumerate} 
\usepackage{caption}
\usepackage{tikz}
\usetikzlibrary{patterns}
\usepackage{stmaryrd}
\usepackage{hyperref}
\usepackage{tcolorbox}
\usepackage{fullpage}
\usepackage{stmaryrd}

\newtheorem{theorem}{Theorem}[section]
\newtheorem{lemma}[theorem]{Lemma}
\newtheorem{cor}[theorem]{Corollary}
\newtheorem{question}[theorem]{Question}
\newtheorem{definition}[theorem]{Definition}
\newtheorem{example}[theorem]{Example}

\newtheorem{notation}[theorem]{Notation}
\newtheorem{proposition}[theorem]{Proposition}
\newtheorem{corollary}[theorem]{Corollary}
\newtheorem{remark}[theorem]{Remark}

\DeclareMathOperator{\Int}{int}

\newcommand\ball[2][]{\mathcal B_{#1}(#2)}
\newcommand{\R}{\mathbb{R}}
\newcommand{\eps}{\varepsilon}
\newcommand{\N}{\mathbb{N}}
\newcommand{\Z}{\mathbb{Z}}

\newcommand\rk\alpha
\newcommand\co[2]{\left\llbracket#1,#2\right\llbracket}
\newcommand\oo[2]{\left\rrbracket#1,#2\right\llbracket}
\newcommand\cc[2]{\left\llbracket#1,#2\right\rrbracket}

\newcommand\cl[1]{\overline{#1}}
\newcommand\restr[1]{_{|#1}}

\DeclareMathOperator{\diam}{diam}

\DeclareMathOperator{\Asym}{Asym}
\newcommand\orb[3][]{\mathcal O^{#1}_{#2}(#3)}
\newcommand{\resp}[1]{\ (resp. #1)}
\newcommand\A{\mathcal A}
\newcommand\B{\mathcal B}
\newcommand\lang[2][]{\mathcal L_{#1}(#2)}
\newcommand\len[1]{\left|#1\right|}
\newcommand{\norm}[1]{\parallel#1\parallel}
\newcommand\card[1]{\left|#1\right|}
\newcommand\pred[3][]{\mathcal P^{#1}_{#2}(#3)}
\newcommand\kk{\textbf k}
\newcommand{\spart}[1]{\left\lceil #1\right\rceil}
\newcommand{\ipart}[1]{\left\lfloor #1\right\rfloor}
\usepackage{xifthen}
\newcommand{\ifnv}[2]{\ifthenelse{\equal{#1}{}}{}{#2}}
\newcommand\sett[3][]{\left\{\left.#2\ifnv{#1}{\in #1}\vphantom{#3}\right|#3\right\}}
\newcommand\defeq{:=}

\newcommand{\pinf}[1]{{}^\infty #1}
\newcommand{\uinf}[1]{#1^\infty}
\newcommand{\dinf}[1]{{}^\infty #1^\infty}
\newcommand\ttau{{\boldsymbol\tau}}
\newcommand{\mirr}[1]{\overline{#1}}
\newcommand\zero{\textbf0}

\author{Silvère Gangloff}
\address[S. Gangloff]{AGH University of Krakow, Faculty of Applied
	Mathematics, al.
	Mickiewicza 30, 30-059 Krak\'ow, Poland
}
\email[S.~Gangloff]{sgangloff@agh.edu.pl}

\author{Pierre Guillon}
\address[P. Guillon]{Institut de Math\'ematiques de Marseille, Bureau 221, Campus de Luminy, Case 907, F-13288 Marseille cedex 9
}
\email[P.~Guillon]{pierre.guillon@math.cnrs.fr}

\author{Piotr Oprocha}
\address[P.~Oprocha]{AGH University of Krakow, Faculty of Applied
	Mathematics, al.
	Mickiewicza 30, 30-059 Krak\'ow, Poland
	-- and --
	Centre of Excellence IT4Innovations - Institute for Research and Applications of Fuzzy Modeling, University of Ostrava, 30. dubna 22, 701 03 Ostrava 1, Czech Republic.}
\email[P.~Oprocha]{oprocha@agh.edu.pl}

\title{Various questions around finitely positively expansive dynamical systems}
\begin{document}
	\maketitle
        \tableofcontents

        \begin{abstract}
It is well-known that when a positively expansive dynamical system is invertible then its underlying space is finite. C.Morales has introduced a decade ago a natural way to generalize positive expansiveness, by introducing other properties that he called positive $n$-expansiveness, for all $n \ge 1$, positive $1$-expansiveness being identical to positive expansiveness. Contrary to positive expansiveness, positive $n$-expansiveness for  $n>1$ does not enforce that the space is finite when the system is invertible. In the present paper we call finitely positively expansive dynamical systems as the ones which are positively $n$-expansive for some integer $n$, and prove several results on this class of systems. 
In particular, the well-known result quoted above is true if we add the constraint of shadowing property, while it is not if this property is replaced with minimality. Furthermore, finitely positively expansive systems cannot occur on certain topological spaces such as the interval, when the system is assumed to be invertible finite positive expansiveness implies zero topological entropy. Overall we show that the class of finitely positively expansive dynamical systems is quite rich and leave several questions open for further research.

%We study here properties of finitely positively expansive dynamical systems (those which are positively $n$-expansive for some integer $n$), in relation with shadowing property and entropy in general. Then we restrict our attention to systems on one-dimensional continua, in which we exhibit different behaviors of these spaces relative to finite positive expansivity: some do not admit systems with this property, while the circle for instance admits some. In the end, we focus on $\Z$-shifts, using substitutions in order to produce many examples of finitely positively expansive systems. 
\end{abstract}

\section{Introduction}

In this text, we study a class of properties of dynamical systems, introduced in 2012 by C.Morales~\cite{Morales}, which generalize 
positive expansiveness, called positive $n$-expansiveness, for $n \ge 1$, and is defined as follows.
%hierarchy of dynamical systems
For every $n \ge 1$ and  a metric space $(X,d)$, a dynamical system $(X,f)$ is called \textbf{positively $n$-expansive} with constant $\eps> 0$ when there exist some $\eps>0$ such that for every $x \in X$, $\card{W_{\eps}(x)}\le n$, where: 
\[W_{\eps}(x) \defeq \sett[X]y{\forall t\in\N, d\left( f^t(x),f^t(y)\right) < \eps}.\]
Let us notice that in particular the classical notion of positive expansiveness is identical to positive $1$-expansiveness.
A dynamical system is called \textbf{finitely positively expansive} when there exists some $n \ge 1$ and $\eps > 0$ such that it is positively $n$-expansive with constant $\eps$. In his article, C.Morales has proved that certain statements which are true for positively expansive dynamical systems 
are still true on finitely positively expansive systems while others are not anymore. It is the case of the following one: 

\begin{theorem} \label{theorem.pos.exp}
Any invertible positively expansive dynamical system is finite. 
\end{theorem}

C.Morales proved that it is not true for finitely positively expansive systems by constructing explicit examples:

\begin{theorem}[Morales]
For each $k\ge 1$, there exists 
an infinite metric space $(X_k,d_k)$ on which there 
is a positively $2^k$-expansive dynamical system $(X_k,f_k)$ such that $f$ is a homeomorphism and which is not $(2^k-1)$-expansive.
\end{theorem}

 During the last decade, the properties %hierarchy 
 defined by C.Morales have received a lot of attention and several formal statements have been proved. Most of them fall into two categories: \textbf{1.} Ones which recover Theorem~\ref{theorem.pos.exp} for systems under additional constraints. The proof relies on the fact that all the properties are equivalent under these constraints. For instance, C.Good, S.Macías, J. Meddaugh, J. Mitchell and J. Thomas~\cite{Good} 
proved that topological mixing, shadowing property and $n$-expansiveness imply together that 
the dynamical system is finite when assumed to be a homeomorphism. B. Carvalho and W. Cordeiro~\cite{CC19} proved a similar statement, where mixing is replaced by transitivity, generalizing the result of A.Artigue, M.J.Pacifico, J.Vieitez~\cite{APV17} that if a  homeomorphism defined on a compact surface is transitive and positively 2-expansive, then it is positively expansive. \textbf{2.} Ones which provide examples of dynamical systems which are homeomorphisms and positively $n$-expansive, for each $n \ge 1$, under various sets of constraints. For instance one can find in~\cite{CC16} an example of homeomorphism with the shadowing property which is $n$-expansive (not positively) and not $(n-1)$-expansive for each $n > 1$. This can be found as well in~\cite{Good}, without the homeomorphism condition. \bigskip

In the present text, we study finitely positively expansive dynamical systems and their properties. 
In Section~\ref{section.general}, we mainly prove two theorems about dynamical systems which are homeomorphisms with the finite positive expansiveness in a general setting: they have zero entropy; if such a system has the shadowing property, it is finite. In Section~\ref{section:continua} we consider dynamical systems on topological graphs. On the circle, there 
are known examples of positively expansive systems, which are the $z^n$ maps for $n \notin \{-1,0,1\}$. We prove that on the other hand, other topological graphs, the interval and the Hawaiian earring, do not have finitely positively expansive systems, for different reasons. In Section~\ref{section:shifts} we consider 
two-sided one-dimensional shifts. In this context, we have found that S-adic 
shifts provide several examples of finitely 
positively expansive dynamical systems. 
In particular, Sturmian and T\oe plitz
shifts are positively $2$-expansive. We provide and prove some sufficient conditions for a sequence of substitutions 
to generate a finitely positively expansive 
shift. Using these results, we are able to 
 strengthen Morales' result by proving that for each $n \ge 2$ there exists a \textbf{minimal} dynamical system which is positively $n$-expansive and not positively $(n-1)$-expansive. \bigskip

We believe that the main outcomes of this study are natural connections between the property of finite positive expansiveness and low complexity symbolic dynamics and dynamics on topological graphs, as witnessed by non trivial results formulated in Section~\ref{section:continua} and Section~\ref{section:shifts}. %We also leave  several open  questions which we discuss Section~\ref{section.open}. 
\bigskip 

In the whole text, we denote by $\N$ the set of non-negative integers and by $\Z_{+}$ the set of positive integers, and $\Z_{-}$ the set of negative integers.

\section{General properties of finitely positively expansive systems}\label{section.general}

In this section, we study finitely positively expansive systems in a general setting. In particular we prove that some properties of dynamical systems are `orthogonal' to finite positive expansiveness when $f$ is a homeomorphism: the shadowing property [Section~\ref{section.shadowing}], and positive entropy  [Section~\ref{section:entropy}]. \bigskip

%\comm{Piotr}{To consider using $O(x)$ and $O^+(x)$. Otherwise it may be hard for the reader to understand some statements correctly, which may lead to improper citations at the end.}

In this section, $(X,d)$ is a compact metric space.
{
  We denote $\ball[\eps]x$ the ball of radius $\eps$ around $x\in X$.
  A \textbf{dynamical system} is a pair $(X,f)$, where $f$ is a continuous self-map of $X$.
}
We will denote by $\orb fx$ the set $\sett{f^k (x)}{k \in \N}$ and call it the \emph{orbit} of $x$ under the action of $f$.
When $f$ is a homeomorphism, we will also denote $\orb[\Z]fx\defeq\sett{f^t(x)}{t \in \Z}$.
When the context is unambiguous, we will simplify these notation into respectively $\orb{}x$ and $\orb[\Z]{}x$.
An $\boldsymbol{\omega}$-\textbf{limit set} of $(X,f)$ is a set of the form 
\[\omega_f(x)\defeq\bigcap_{n \ge 1} \overline{\bigcup_{k \ge n} \left\{f^k(x)\right\}},\]
where $x \in X$. We say that the dynamical system $(X,f)$ is \textbf{transitive} when for any open sets $u,v$ of $X$, there exists an integer $n>0$ such that $f^n(u)\cap v\neq \emptyset$.

\subsection{Elementary properties}

%Let us recall the definition of finitely positively expansive dynamical system. We also prove in this section some direct consequences of the definitions. 

%Let us consider a dynamical systems $(X,f)$.

%\begin{notation}
%For every $x \in X$ and $\eps > 0$, let us denote by $W_{\eps}(x)$ the following set: 
%\[W_{\eps}(x) \defeq \sett[X]y{\forall k \in\N, \ d\left( f^k(x),f^k(y)\right) < \eps}.\]
%\end{notation}

%\begin{definition}
%The dynamical system $(X,f)$ is said to be \textbf{positively $n$-expansive} if there exists some $\eps>0$ such that for every $x \in X$, $|W_{\eps}(x)| \le n$. Furthermore, it is said to be \textbf{finitely positively expansive}
%when there exists some $n\geq 1$ such that it is positively $n$-expansive. 
%\end{definition}

Let us notice first that the class of finitely positively expansive dynamical systems is stable by product, union, and subsystem: 

\begin{remark}\label{remark:elem}
For two dynamical systems %$(X,f)$ 
%and 
%$(Y,g)$ 
which are, respectively, positively $n$-expansive and positively $m$-expansive, we have the following:
	\begin{enumerate}
		\item The disjoint union %$(X\sqcup Y,f\sqcup g)$ 
		of these dynamical systems is positively $\max(m,n)$-expansive;
		\item Their product %$(X\times Y,f\times g)$ 
		is positively $mn$-expansive.
		\item Any iteration of the first one %$(X,f^t)$ 
		is positively $n$-expansive.
		\item Any of its subsystems %$(X',f_{X'})$ 
		is positively $n$-expansive.
\end{enumerate}\end{remark}

%As well the following is straightforward:
%\begin{lemma}\label{lemma.iterate.expansive}
%	Let us assume that $(X,f)$ is finitely positively expansive. Then for every $n \ge 1$, $(X,f^n)$ is also  finitely positively expansive. 
%\end{lemma}

\begin{lemma}\label{l:locally.preimages}
A system $(X,f)$ which is positively $n$-expansive is locally at-most-$n$-to-one (i.e. there exists $\eps>0$ such that for every $x \in X$,  and all $z\in X$, $\card{f^{-1}(z) \cap\ball[\eps]x} \le n$).
\end{lemma}
\begin{proof}
Since $(X,f)$ is positively $n$-expansive, there exists $\eps >0$ such that for every $x \in X$, $|W_{2\eps}(x)|\le n$.
We claim that for every $x \in X$, $z \in X$ and $z' \in f^{-1}(z) \cap \ball[\eps]x$, 
\[f^{-1}(z) \cap \ball[\eps]x \subset W_{2\eps}(z').\]
Indeed, consider $z'' \in f^{-1}(z) \cap \ball[\eps]x$. By the triangular inequality, 
\[d(z'',z') \le d(z'',x) + d(x,z') \le 2\eps.\]
Furthermore, since $f(z'')=f(z') = z$, for every $n \ge 1$, we have $d(f^n(z''),f^n(z'))=0 \le 2 \eps$. Therefore $z'' \in W_{2\eps}(z')$. Since this is true for every $z'' \in f^{-1}(z) \cap \ball[\eps]x$, we have $f^{-1}(z) \cap \ball[\eps]x \subset  W_{2\eps}(z')$.
%{\blue Fix any $z'\in f^{-1}(z) \cap \ball[\eps]x$. Then $f^{-1}(z) \cap \ball[\eps]x \subset W_{2\eps}(z')$.}
This shows that the claim holds and therefore we have $|f^{-1}(z) \cap \ball[\eps]x| \le n$.  This ends the proof.
\end{proof}

\begin{corollary}\label{cor.bounded.fibers}
A system $(X,f)$ which is finitely positively expansive is bounded-to-one (i.e. the function $x \mapsto\card{f^{-1}(x)}$ is bounded).
\end{corollary} 
\begin{proof}
Considering some $\eps>0$ provided by Lemma~\ref{l:locally.preimages}, and a finite cover $\mathcal{U}$ by balls of the form $\ball[\eps]x$, Lemma~\ref{l:locally.preimages} gives that for every $z \in X$, $\card{f^{-1}(z)} \le n \card{\mathcal{U}}$.
%Let us consider $n$ such that $(X,f)$ is positively $n$-expansive, which means 
%that there exists $\eps$ such that for every $x \in X$, $|W_{\eps}(x)| \le n$.
%We know that there exists $N>0$ such that if $\mathcal{U}$ is a set of open balls of diameter $\eps$ such that for every $B,B' \in \mathcal{U}$, the center of $B$ is not in $B'$, then 
%$|\mathcal{U}| \le N$.
%For every $x \in X$, there exists $\mathcal{U}_x$ set of balls of diameter $\eps$ centered on elements of %$f^{-1}(x)$ which covers $f^{-1}(x)$. 
%We have that $|\mathcal{U}_x| \le N$.
%Furthermore, each of the balls in this set contains at most $n$ elements of $f^{-1}(x)$ (otherwise the positive $n$-expansivity would be contradicted). Hence $|f^{-1}(x)| \le n |\mathcal{U}_x| \le  n N$.
\end{proof}

%For a dynamical system $(X,f)$, the suspension space is the space $X \times [0,1] / \sim$, where $\sim$ is defined 
%by $(x,t) \sim (f(x),t-1)$
%for every $x,t$. The flow on this space is the natural one.

%\begin{definition}
%We say that two dynamical systems are \textbf{flow equivalent} when their suspension flows are equivalent. 
%\end{definition}

%\begin{definition}
%Two systems are \textbf{orbit equivalent} when there is a homeomorphism which
%put orbits into bijection.
%\end{definition}

%The following is known:

%\begin{proposition}
%Whenever two dynamical systems $(X,f)$ and $(Y,g)$ are conjugates, they are orbit equivalent, which implies that they are flow equivalent.
%\end{proposition}

\begin{proposition}
For every integer $n \ge 1$, any dynamical system conjugate to a positively $n$-expansive dynamical system is also positively $n$-expansive.
\end{proposition}

\begin{proof}
Let us consider two dynamical systems $(X,f)$ and $(Y,g)$ respectively on metric spaces $(X,d)$ and $(Y,d')$. Assume that $(X,f)$ 
is positively $n$-expansive and consider $\eps>0$ such that for every $x \in X$, $|W_{\eps}(x)|\le n$. Let us also consider $\sigma : X \rightarrow Y$ a 
homeomorphism such that $\sigma \circ f \circ \sigma^{-1} = g$. 
Since $\sigma^{-1}$ is continuous, there exists some $\epsilon>0$ such that if $d(x,y)<\epsilon$ then $d'(\sigma^{-1}(x),\sigma^{-1}(y))< \eps$. Then for every $y \in Y$, 
\[W_{\epsilon}(y) = \sett[Y]z{\forall t\in\N,d(\sigma \circ f^t(\sigma^{-1}(z)), \sigma \circ f^t(\sigma^{-1}(y)))<\epsilon} \subset \sigma\left(W_{\eps}(\sigma^{-1}(y))\right).\]
As a consequence $\card{W_{\epsilon}(y)}\le n$. Since this is verified for every $y$, this means that $(Y,g)$ is positively $n$-expansive.
\end{proof}

Two elements $x,y$ of $X$ are \textbf{asymptotic} when 
$d(f^n(x),f^n(y)) \rightarrow 0$ when $n \rightarrow +\infty$. This defines an equivalence relation, and we denote by $\Asym(x)$ the equivalence class of $x$ for this relation. Clearly, for every $x$ and $\eps >0$, and any integer $n$ such that $n\le\card{\Asym(x)}$, there exists some $t\in\N$ such that $\card{W_{\eps}(f^t(x))}\ge n$. In other words, we have the following.

\begin{remark}\label{rem:asy}
If a system $(X,f)$ which is a homeomorphism is positively $n$-expansive, then for every $x \in X$, we have $\card{\Asym(x)}\le n$.
\end{remark}
\begin{example}
It is straightforward to see that even if a factor of a system is positively finitely expansive, the original system is not necessarily positively finitely expansive: one can consider for instance that the shift $\{\dinf0\}$ is a factor of the full shift $\{0,1\}^{\Z}$.
%On the other hand, we will see in Section~\ref{section:minimal} that a factor of a finitely positively expansive shift is not necessarily finitely positively expansive.
\end{example}

The above example leads by Remark~\ref{rem:asy} to the following simple observation.
\begin{cor}\label{cor:periodic}
Suppose that an invertible dynamical system $(X,f)$ is finitely positively expansive, and $x\in X$ is not periodic.
Then $\Asym(x)$ does not contain any periodic orbit.
\end{cor}

\subsection{Equicontinuity}

{
\begin{definition}
  Let $(X,f)$ be a dynamical system.
  If $\eps>0$, a point $x$ is \textbf{$\eps$-stable} if there exists $\delta_x>0$ such that for every $y\in X$ such that $d(x,y) < \delta_x$ we have $d(f^t(x),f^t(y)) < \eps$ for every $t\in\N$.
  A point $x$ is \textbf{equicontinuous} if it is $\eps$-stable for every $\eps>0$.
  The system $(X,f)$ is said to be \textbf{equicontinuous} when all points are equicontinuous (in which case compactness allows to chose $\delta_x$ uniformly).
  If $\eps>0$, $(X,f)$ is said to be \textbf{$\eps$-sensitive} when no point is $\eps$-stable. It is \textbf{sensitive} if it is $\eps$-sensitive for some $\eps>0$.
\end{definition}

% \begin{remark}[\cite{kurka?}]
%   Consider a system $(X,f)$ and an equicontinuous point $x$.
%   Every point $y\in X$ such that $x\in\cl{\orb fx}$ is also equicontinuous.
% \end{remark}

%\begin{definition}
In the following, we will say that a dynamical system $(X,f)$ is \textbf{finite}\resp{\textbf{countable}, \textbf{uncountable}} when $X$ is so.
% \end{definition}
\begin{proposition}\label{p:dicho}
  Every finitely positively expansive system with constant $\eps>0$ is the disjoint union of a discrete set of equicontinuous points and a (perfect) $\eps$-sensitive subsystem.
\end{proposition}
\begin{proof}
  Let $x\in X$ be $\eps$-stable: the ball of radius $\delta_x$ around $x$ is included in the finite set $W_\eps(x)$.
  This implies that $x$ is isolated, hence equicontinuous.
  By continuity of $f$, the set of $\eps$-unstable points is invariant by $f$.
  Moreover, it is closed because its complement is a union of isolated points, hence it yields a $\eps$-sensitive subsystem.
\end{proof}
% + counter-example where the first part is not a subsystem: $\uinf0\cdot$Thue-Morse.
\begin{corollary}\label{c:equicontinuous}~
  \begin{enumerate}
    \item\label{i:cleq} In any finitely positively expansive system, every closed set of equicontinuous points is finite.
    \item\label{i:eq} Any finitely positively expansive equicontinuous system is finite.
    \item\label{i:sens} Any finitely positively expansive system $(X,f)$ over a perfect space $X$ is sensitive.
    \item Any finitely positively expansive system admits only finitely many periodic points of each period $p\in\Z_+$.
    \item\label{i:indnb} Any finitely positively expansive system is either finite or uncountable.
\end{enumerate}\end{corollary}
\begin{proof}~\begin{enumerate}
  \item Every closed discrete set is finite.
  \item This is clear from \eqref{i:cleq}.
  \item This is clear.
  \item The restriction of $f$ to the set of periodic points of period $n$ is an equicontinuous subsystem.
  \item It is folklore that every non-empty sensitive dynamical system is uncountable (as is every non-empty perfect Hausdorf space).
    \popQED\end{enumerate}\end{proof}

\subsection{Transitivity and periodicity}
A dynamical system $(X,f)$ is said to be \textbf{aperiodic} when it has no periodic point.
We also say that a point $x$ is \textbf{preperiodic} if its orbit $\orb fx$ is finite% there exist $p\in\Ns$ and $q\in\N$ such that $f^{p+q}(x)=f^q(x)$
, and \textbf{asymptotically periodic} if there exist an integer $p\ge 1$ and $y\in X$ such that $\lim_{t\to\infty}f^{pt}(x)=y$.
\begin{proposition}~\label{p:transnexp}
  Assume that $X$ is such that all open balls of radius $\eps$ are closed, and that $(X,f)$ is an invertible finitely positively expansive system with constant $\eps>0$. % is such that $\orb fx$ is dense in $X$.
  If $x\in X$ is not periodic, then $\cl{\orb fx}$ is aperiodic. %or it is equal to $\orb fx$ and $x$ is periodic.
  In particular, every asymptotically periodic point is periodic.
\end{proposition}
The assumptions	of the proposition above are satisfied when $X$ is totally disconnected, or more generally when $\eps$ is larger than the diameter of every connected component.
%\comm{Pierre}{i removed the invertibility assumption; it still works, no?}
%\comm{Piotr}{The proof is wrong. Finite orbit does not imply periodic for noninvertible case. In invertible case there is no pre-periodic on the other hand. So reverting to old version is not straight.}
\begin{proof}
  Assume that $\cl{\orb fx}$ contains a point $y$ with some period $p\in \Z_+$.
  First, suppose that there exists $q\in\co0p$ such that for every $t\in \N$, $d(f^t(x),f^{q+t}(y))<\eps$.
  In particular, for every $k\in\N$, $d(f^{kp+t}(x),f^{q+t}(y))<\eps$.
  This means that $\orb {f^p}x\subset W_\eps(f^q(y))$.
  Since this set is finite, $x$ is periodic, which is a contradiction. %preperiodic.
  Now, suppose that for every $q\in\co0p$, there exists $t\in \N$ such that $d(f^t(x),f^{q+t}(y))\ge\eps$.
  Since $y\in\cl{\orb fx}$, for every $\ell>t$, there exists $t_\ell$ such that $f^{t_\ell}(x)$ is in the open neighborhood $\bigcap_{i=0}^{\ell-1}f^{-i}(\ball[\eps]{f^{q+i}(y)})$ of $f^q(y)$, for some $q\in\co0p$.
  Thanks to our assumption, one can assume $t_\ell>t$ to be minimal, so that $d(f^{t_\ell-1}(x),f^{q-1\bmod p}(y))\ge\eps$.
  By compactness of $X$ (and finiteness of $\co0p$), there exists $q\in\co0p$ such that $f^{t_\ell-1}(x)$ admits a limit point $z$ such that $d(z,f^{q}(y))\ge\eps$ and for every $i\in \N$,  $d(f^i(z),f^{q+i}(y))\le\eps$.
    But by assumption, $\ball[\eps]{f^q(y)}$ is a closed set, hence the last inequality becomes $d(f^i(z),f^{q+i}(y))<\eps$.
  Again, this means that $\orb {f^p}z\subset W_\eps(f^q(y))$.
  Since this set is finite, $z$ is periodic, which contradicts that $d(f^{-1}(z),f^{q-1}(y))\ge\eps$.
\end{proof}
The conclusion of Proposition~\ref{p:transnexp} does not imply that, in the case when $f$ is invertible, the closure of the twosided orbit $\cl{\orb[\Z]fx}$ is aperiodic, though.
For instance, consider the configuration $x$ such that $x_{\Z_{-}}=\pinf0$ and $x_{\N}$ is the concatenation of all finite iterates of the Thue-Morse substitution from symbol $0$.

\begin{remark}
  Proposition~\ref{p:transnexp} implies that transitive finitely positively expansive systems over a totally disconnected space are aperiodic or finite.
  Indeed, every transitive dynamical system admits a point $x \in X$ such that $\orb fx$ is dense.
  This is actually a characterization of transitivity, in the case when $f$ is surjective or when $X$ is perfect (see for instance in \cite{Kolyada}).
\end{remark}

}
\subsection{Entropy}\label{section:entropy}
While finitely positively expansive invertible dynamical systems can be act on infinite spaces, they must have low complexity. Strictly speaking, after presenting some definitions and known facts in Section~\ref{section.entropy.def}, we prove in Section~\ref{section.entropy.thm} that every such system has zero topological entropy.

\subsubsection{Definitions}\label{section.entropy.def}

In this section we recall a few facts about topological entropy which will be used later. We direct the reader who is not familiar with this notion to any
standard textbook on ergodic theory, e.g. \cite{Walters}.

Any set of subsets of $X$, denoted by $\mathcal{U}$, is a \textbf{cover} of $X$ if: 
\[
\bigcup_{u \in \mathcal{U}} u = X.
\]
A cover is said to be \textbf{open} when all its elements are open sets. Provided a cover $\mathcal{U}$, we denote, for every $n \in\N$, by $\bigvee_{i=0}^{n-1}f^{-i}\mathcal{U}$, the set of all non-empty subsets of $X$ of the form: 

\[u_0 \cap f^{-1}(u_1) \cap \ldots \cap f^{-{n-1}} (u_{n-1}),\]

where for every $i \in \llbracket 0 , n \llbracket$, $u_i \in \mathcal{U}$. Clearly, for every $n \in\N$, $\bigvee_{i=0}^{n-1}f^{-i}\mathcal{U}$ is also a cover of $X$ and it is open when $\mathcal{U}$ was open. 

\begin{notation}
For every $K \subset X$ compact and $\mathcal{U}$ open cover of $K$, we denote by 
$\mathcal{N}(K,\mathcal{U})$ the minimal number of elements of $\mathcal{U}$ which form
a cover of $K$ (because $K$ is compact, this number is well defined). When $K=X$, we simplify the notation $\mathcal{N}(X,\mathcal{U})$ and simply write $\mathcal{N}(\mathcal{U})$.
\end{notation}

%For any cover $\mathcal{U}$, we denote by $N(\mathcal{U})$ the minimal cardinal of a subset of $\mathcal{U}$ which is also a cover. 

The sequence $n \mapsto \mathcal{N}\left(\bigvee_{i=0}^{n-1}f^{-i}\mathcal{U}\right)$ is sub-multiplicative. As a consequence, the sequence 
$n \mapsto \log_2 \left(\mathcal{N}\left(\bigvee_{i=0}^{n-1}f^{-i}\mathcal{U}\right)\right)/n$
converges. We denote this limit by $h(\mathcal{U},X,f)$ and call it \textbf{entropy of the cover $\mathcal{U}$}.
The \textbf{topological entropy} $h(X,f)$ of the system $(X,f)$ is defined as the supremum, over finite open covers $\mathcal{U}$ of $X$, 
of $h(\mathcal{U},X,f)$.

For any open cover $\mathcal{U}$, we call \textbf{closure} of $\mathcal{U}$ 
and denote it by $\overline{\mathcal{U}}$, the set of $\overline{u}$ for $u \in \mathcal{U}$.
The following is straightforward:
\begin{lemma}\label{lemma.entropy.closed.cover}
The topological entropy of $(X,f)$ is also equal to the supremum of $h(\overline{\mathcal{V}},X,f)$, where $\mathcal{V}$ ranges over the open covers. 
\end{lemma}
%
%\begin{proof}
%Let us consider any finite open cover $\mathcal{U}$ of $X$. For every $x \in X$, there exists some $\eps_x >0$ and an element $u_x$ of $\mathcal{U}$ such that $\overline{B_{\eps_x}(x)} \subset u_x$. The set $\{B_{\eps_x}(x) : x \in X\}$ is an open cover 
%and thus admits, by compactness of $X$, a finite subset $\{B_{\eps_{x_k}}(x_k) : 0 \le k \le m\}$ which is also a cover. As a consequence the set $\mathcal{V}$ of balls $\overline{B_{\eps_{x_k}}(x_k)}$, $0 \le k \le m$ is also a cover of $X$. 
%For every $n$, $N(\delta_n \mathcal{V}) \ge N(\delta_n \mathcal{U})$ and thus 
%$h(\mathcal{V}) \ge h(\mathcal{U})$. 
%
%Reciprocally, for every $\mathcal{U}$ open cover of $X$, we have 
%\[h(\overline{\mathcal{U}})\le h(\mathcal{U}).\]
%
%The conclusion is straightforward. 
%\end{proof}

A set $E\subset X$ is $\boldsymbol{(n,\epsilon)}$-\textbf{separated}, if for any distinct $x,y\in E$ there is some $0\leq i <n$ such that $d(f^i(x),f^i(y))>\epsilon$.
Let us denote $s(n,\epsilon)$ the %number 
maximal cardinality of an $(n,\epsilon)$-separated set. 
We will also need the following well known fact, proven first by Bowen: 

\begin{theorem}\label{thm:separated.sets}
The entropy of the system $(X,f)$ is given 
by the following formula: 
\[h(X,f) = \lim_{\epsilon \rightarrow 0} \limsup_n \frac{1}{n} \log(s(n,\epsilon))\]
\end{theorem}

%If $X$ is a shift, then for simplicity we will denote by $h(X)$ entropy of $\sigma|_X$. Furthermore, we will denote $\mathcal{N}_n(X)$ the number of globally admissible words of length $n$ of this shift for every $n \ge 1$. It is well known that (e.g. see \cite{Marcus}): 
%\[h(X) = \lim_n \frac{\log(\mathcal{N}_n(X))}{n}.\]

%\comm{Silvere}{Put this paragraph in the part related to shifts?}

\subsubsection{Finitely positively expansive systems have zero entropy}\label{section.entropy.thm}

\begin{lemma}\label{lemma.decreasing.compacts}
	Let us consider any finite open cover $\mathcal{U}$ of $X$, 
	and $K_n$ a decreasing sequence of compact subsets of $X$. Then $\mathcal{N}(K_n,\mathcal{U})$
	converges towards $\mathcal{N}(K,\mathcal{U})$, where 
	\[K = \bigcap_n K_n.\]
\end{lemma}

\begin{proof}
	Clearly for every $n$, $\mathcal{N}(K_n,\mathcal{U}) \ge \mathcal{N}(K,\mathcal{U})$ and the sequence 
	$n \mapsto \mathcal{N}(K_n,\mathcal{U})$ is non-increasing. It is thus sufficient to prove 
	that for some $n$, $\mathcal{N}(K_n,\mathcal{U}) = \mathcal{N}(K,\mathcal{U})$. Let us assume that 
	for every subset $\mathcal{V}$ of $\mathcal{U}$ of cardinality $k\defeq \mathcal{N}(K,\mathcal{U})$  and for every $n$, there exists some point $x_n$ of $K_n$ outside of the set
	\[\mathcal{E}(\mathcal{V}) \defeq \bigcup_{u \in \mathcal{V}} u.\]

Fix any $\mathcal{V}\subset \mathcal{U}$ of cardinality $k\defeq \mathcal{N}(K,\mathcal{U})$ and  $x_n\in K_n$ as above.
 By compactness, there exists a subsequence of $(x_n)$ which converges towards some $x \in X$ and clearly $x \in K$. Since $\mathcal{E}(\mathcal{V})$ is open, its complement is closed. As a consequence, $x \in X\setminus \mathcal{E}(\mathcal{V})$. 
	Hence $K$ is not contained in $\mathcal{E}(\mathcal{V})$. Since it is true for every 
	$\mathcal{V}$, this means that $\mathcal{N}(K,\mathcal{U})  > k$ which is a contradiction. 
\end{proof}

\begin{theorem}\label{thm:entropy.zero.gen}
	Any invertible finitely positively expansive system $(X,f)$ has zero entropy. 
\end{theorem}

\begin{proof}
	Let $N$ be such that the system $(X,f)$ is positively $N$-expansive, and $\eps >0$ such that for every $x$, $|W_{\eps}(x)| \le N$. 
	\begin{enumerate}
	    \item \textbf{Notation simplification.} 	In this proof, in order to simplify a bit the formula, let us set 
    \[\mathcal{U}^n \defeq \bigvee_{i=0}^{n-1}f^{-i}\mathcal{U}\]
    for every $n$ and all cover $\mathcal{U}$.
\item \textbf{Claim.} Let us consider any open cover $\mathcal{U}$ whose every element
%	is included in some open ball of 
has diameter smaller than $\eps/2$, and fix any integer $n \ge 1$. 
	We claim that there exists some $m_n$ such that for every $m \ge  m_n$, and every sequence $u_0, \ldots u_m$ of elements of $\overline{\mathcal{U}}$, we have: 
	\[
	\mathcal{N}(f^{-n} (u_0) \cap f^{-1-n}(u_1) \cap \ldots \cap f^{-m-n}(u_m),\mathcal{U}^n) \le N.
	\]
	\item \textbf{Proof of the claim.} If the claim were false, there would exist an integer $n$ and an infinite sequence $v_0, v_1 , ... \in \overline{\mathcal{U}}$ 
	such that for every $m_n$ there is $m>m_n$ such that
	
	\[\mathcal{N}(f^{-n} (v_0) \cap f^{-1-n}(v_1) \cap \ldots \cap f^{-m-n}(v_m),\mathcal{U}^n) \ge N+1.\]
	By Lemma~\ref{lemma.decreasing.compacts}, we have 
	\[\mathcal{N}(K,\mathcal{U}^n) \ge N+1,\]
	where 
	\[K = f^{-n} \left(\bigcap_{i \in\N} f^{-i} (v_i)\right).\]
	However, by positive $N$-expansiveness and because $f$ is a homeomorphism, $|K| \le N$. Hence $\mathcal{N}(K,\mathcal{U}^n) \le N$, which is a contradiction. This proves the claim.
	\item \textbf{Consequences.} We deduce from the above claim that for every $m \ge m_n$, and for any element $u$ of $\overline{\mathcal{U}^m}$, there exists a set of $N$ elements of $\mathcal{U}^n$ which cover $f^{-n}(u)$, and thus there are $N$ elements of $\overline{\mathcal{U}^n}$ which also cover $f^{-n}(u)$.
	This implies: 
	
	\[\mathcal{N}(\overline{\mathcal{U}^{n+m}}) \le N \cdot \mathcal{N}(\overline{\mathcal{U}^m}).\]
	
	Since this is satisfied for every $m \ge m_n$, for every $k \ge 1$, we have: 
	\[
	\mathcal{N}(\overline{\mathcal{U}^{m+kn}}) \le N^k \mathcal{N}(\overline{\mathcal{U}^m}).
	\] 
	With an application of the logarithm function: 
	\[\frac{\log(\mathcal{N}(\overline{\mathcal{U}^{m+kn}}))}{m+kn} \le \frac{k}{m+kn}\log(N) + \frac{\log(\mathcal{N}(\overline{\mathcal{U}^m}))}{m+kn}.\]
	By taking $k \rightarrow +\infty$, we obtain:
	\[h(\overline{\mathcal{U}}) \le \frac{\log(N)}{n}.\]
	
	Since $n$ was taken arbitrary, $h(\overline{\mathcal{U}}) = 0$. 
		As a direct consequence of Lemma~\ref{lemma.entropy.closed.cover}, the topological entropy 
	of $(X,f)$ is equal to zero. 
	\popQED\end{enumerate}
\end{proof}
{
In contrast, finitely positively expansive systems which are not homeomorphisms can have arbitrary entropy value: consider for instance one-sided shifts, which are all positively $1$-expansive.
}

\subsection{Shadowing property}\label{section.shadowing}

It is known that an invertible dynamical system which is positively expansive is finite. The situation is 
different for the property of finite positive expansiveness. For instance we will see some counter-examples in Section~\ref{section:shifts} (in the context of symbolic dynamics). On the other hand, we can recover the same result as for positive expansiveness 
by adding the constraint of having \textit{shadowing property}.

\begin{definition}
For a dynamical system $(X,f)$ on a metric space $(X,d)$, an $\epsilon$-\textbf{pseudo-orbit} is a sequence $(x_n)_{n \in\N}$ such that for every $n \in\N$, 
\[d(f(x_n),x_{n+1}) < \epsilon.\]

If for $p,q\in X$ there is $m>0$ and a finite sequence $(x_0,\ldots,x_m)$ such that for each $n=0,\ldots,m-1$
\[d(f(x_n),x_{n+1}) < \epsilon.\]
and $x_0=p$, $x_m=q$ then we say that there is an \textbf{$\epsilon$-chain} from $p$ to $q$. If for every $\epsilon>0$ there is an $\epsilon$-chain between any two points $p,q$ then the system $(X,f)$ is chain-transitive. 

%For a point $p\in X$, if every $\eplison>0$ there is an $\eps$-chain from $p$ to $p$ then we say that $p$ is chain-recurrent %point. If every point is 
\end{definition}

\begin{definition}
A dynamical system $(X,f)$ has the \textbf{shadowing property} when for every $\epsilon > 0$, there exists some $\delta > 0$ such that every $\epsilon$-pseudo-orbit $(x_n)$ is $\delta$-\textbf{shadowed} by some $x \in X$ such that for every $n \in\N$, 
\[d(x_n,f^n(x)) < \epsilon.\]
\end{definition}

The following theorem %provides a converse implication to Proposition \ref{prop.equicontinuous} under shadowing property. It
is contained in Theorem~{3.7} in \cite{LOWM}: 
\begin{theorem}\label{thm:sensitive.point}
A dynamical system with the shadowing property has positive entropy if and only if it has a non-equicontinuous point.
\end{theorem}

% Assume that the system $(X,f)$ is not equicontinuous.
% Then there is a point at which $(X,f)$ is not equicontinuous.  
% By Theorem~\ref{thm:sensitive.point}, $h(X,f)>0$. 
% By the variational principle, there is an ergodic measure 
% $\mu$ such that $h_{\mu} (X,f) >0$. 
% %Since $f$ restricted to $\supp \mu$ is transitive, it must be minimal on that set. T
% By \cite[Theorem~4.3]{LO} there is at least one minimal subsystem $Y\subset X$ with positive topological entropy.
% Therefore, we obtain the following lemma as a direct consequence of results in \cite{LO}.
% \begin{lemma}\label{p:minposh}
% Any non-equicontinuous system with the shadowing property admits a minimal positive-entropy system.
% \end{lemma}

Now we can recover \cite[Theorem~G]{Vieitez}, with another proof, based on relations between shadowing property and entropy.
\begin{theorem}\label{theorem.shadowing}
Any invertible finitely positively expansive system with the shadowing property is finite.
\end{theorem}

%Before that we will need to list some results that we will use in the proof 
%as well as some definitions. 
%\subsection{Proof of Theorem~\ref{theorem.shadowing}}

\begin{proof}
    From Theorem~\ref{thm:entropy.zero.gen}, such a system must have entropy $0$.
    From Theorem~\ref{thm:sensitive.point}, this implies that the system is equicontinuous.
    We can conclude by Corollary~\ref{c:equicontinuous}\eqref{i:eq}.
\end{proof}

%\section{Generalizations of Theorem~\ref{theorem.shadowing} for zero-dimensional systems}

\section{Dynamical systems on one-dimensional continua}\label{section:continua}

It is a well known fact that the interval has no positively expansive dynamical system (e.g. see \cite{ALM,AH}). In this section, we prove that it also has no finitely positively expansive one. On the other hand, we know that there are positively expansive systems on the circle ($z^n$ maps for $n \notin \{-1,0,1\}$). This observation naturally leads to the following questions. 

\begin{question}
Can we find, for every $n\ge 1$, a positively $(n+1)$-expansive map of the circle which is not positively $n$-expansive?
\end{question}

\begin{question}\label{question:continua}
  Which one-dimensional continua admit finitely positively expansive systems?
\end{question}

We leave the first question open, but provide some results in the direction of the second one. 

%\subsection{There is no finitely positively expansive system on the interval}
\subsection{Systems over the interval}
\label{section:interval}

The aim of this section is to prove Theorem~\ref{theorem:interval}.
% stated later, in Section~\ref{section.interval.proof}. We first expose some preliminary definitions and known results in Section~\ref{section.interval.preliminaries} which we will use in the proof.
For the remainder of this section, we set $I \defeq[0,1]$. %and fix a continuous map $f\colon I\to I$. 

A dynamical system $(X,f)$ is said to be \textit{mixing} when there exists an integer $N$ such that 
for every non-empty open subsets $u,v$ of $X$, and for every $n \ge N$, $f^n (u) \cap v \neq \emptyset$.

\begin{proposition}\label{p:sensmix}
  For every sensitive system $(I,f)$ over the interval, there exists a non-degenerate subinterval $J\subset I$ and some $k>0$ such that $f^k\restr J$ is mixing.
  Moreover, it has positive entropy.
%	If an interval system $f\colon I\to I$ has zero topological entropy, then it is not finitely positively expansive.
\end{proposition}
\begin{proof} 
  From \cite[Theorem~1.3]{blokh} (see \cite[Proposition~2.40]{Ruette} for a proof), every sensitive system of the interval admits a non-degenerate (so-called \emph{periodic}) interval $J'$ and $p>0$ such that $f\restr{\bigsqcup_{0\le i<p}f^i(I)}$ is a transitive subsystem.
  This implies that $f^p\restr I$ is a transitive system over interval $J$.
  Moreover, \cite[Propositions~2.16 and~2.17]{Ruette} imply that $f^{2p}\restr{J}$ is mixing, for some non-degenerate interval $J\subset I$.
  By \cite[Lemma~1.2]{blokh} (see \cite[Proposition~4.70]{Ruette}), any transitive system of some non-degenerate interval has positive entropy.
\end{proof}

%Recall, that if an interval system has positive entropy, then it has a point of odd period.

\begin{remark}\label{r:fnposexp}
  Let $(X,f)$ be a dynamical system, $n\ge 1$, $\eps>0$, $J\subset X$ with cardinality $n+1$, and $t$ such that:
  \[\forall j\in \cc0{t},\diam f^j(J) <\eps
  \text{ and }\card{f^t(J)}=1.\]
Then $(X,f)$ is not positively $n$-expansive with constant $\eps$.
Indeed, $J\subseteq W_\eps(x)$, for any $x\in J$.
\end{remark}
% \begin{proof}
%   By the second assumption, for every $y\in f^{t_k}(J_{k,0})$, there exist $k$ distinct points $x_0,\ldots,x_{k-1}$, respectively in the intervals $J_{k,0},\ldots, J_{k,k-1}$ such that for every $i$, $f^{t_k}(x_i)=y$.
%   By the first assumption, for every $i\in\co0k$ and every $t\in\co0{t_k}$, $d(f^t(x_i),f^t(x_0))<\eps$.
%   Putting the two phases together, we obtain that $x_i \in W_{\eps}(x_1)$, which implies that $\card{W_{\eps}(x_1)}\ge k$.
% \end{proof}

% \begin{proposition}
%   Let $(X,f)$ be a positively $n$-expansive system with constant $\eps>0$, and $J\subset I$ a subset with diameter less than $\eps$.
%   Let $\S_t\defeq\bigcap_{j\le t}f^{-j}(J)$ the points staying in $J$ for $t$ steps. %, and $\E\defeq\sett[\N]t{\S_{t-1}\subsetneq\S_t$ its set of \textbf{escaping times}.
%   Then there exists $t\le n$ such that $\card{\S_n}\le n-t$ and $\forall m\ge t,\S_m=\S_t$.
% \end{proposition}
% \begin{proof}
  
% \end{proof}

{ 
\begin{lemma}\label{lem.interm}
  Let us consider an interval system $f \colon I \to I$, and $c$ in the interior of $f(I)$.
  For every $\eps>0$, there exists an interval $J \subset I$ such that $\diam J < \eps$ such that $c$ is in the interior of $f(J)$.
\end{lemma}

\begin{proof}
Let us assume ad absurdum that there exists $\eps > 0$ such that for all $J \subset I$ such that $\diam J < \eps$, 
$c$ is not in the interior of $f(J)$.  This implies that for all $x \in f^{-1}(c)$, there exists some $J \subset I$ such that 
$x \in J$ and $f(J) \le c$ or $f(J) \ge c$. Let us denote by $\delta^{+}$ (resp. $\delta^{-}$) the set of $x \in f^{-1}(c)$
such that there exists some $J$ containing $x$ such that $f(J) \ge c$ (resp. $f(J) \le c$). If any of the sets $\delta^{+},\delta^{-}$ was equal to $f^{-1}(c)$, then we would have $f(I) \le c$ or $f(I) \ge c$ which was assumed to be false ($c$ is in the interior of $f(I)$). Without loss of generality, we can assume that there exists $x \in \delta^{-}, x' \in \delta^{+}$ such that $x < x'$.
Let us then denote by $y$ the supremum over the elements of $\delta^{-}$ which smaller than $x'$, and $z$ the 
infimum over the elements of $\delta^{+}$ which are greater than $y$. Then else $y < z$, and by the intermediate value theorem one can find some $x'' \in f^{-1}(c)$ such that $y < x'' < z$, and thus can't be in $\delta^{-}$ or $\delta^{+}$, which is impossible, 
or $y=z$ which means that there exists some $J$ containing $z$ in its interior such that $f(J) = \{c\}$, which contradicts the hypothesis ad absurdum. This concludes the proof.
\end{proof}

\begin{remark}\label{r:interval.preimage}
  For every interval continuous map $f \colon I \to I$, and every interval $J \subset f(I)$, there exists a subinterval $I'\subset I$ such that $f(I') = J$.
\end{remark}
}

\begin{lemma}\label{l:mix}
  If an interval system $f\colon I\to I$ is mixing, then it is not finitely positively expansive.
\end{lemma}
\begin{proof} %Let us consider the following intermediate steps.
  Since $f$ is mixing, $f$ is not monotone.
  As a consequence, $f$ has a local extremum $c\in\Int I$. 
  Let us fix some $\eps>0$ and prove, by induction over $k\in\N$, that 
  there exist $2^k$ disjoint intervals $J_{k,1},\ldots,J_{k,2^k}$ and an integer $t_k$ such that:
  \[\forall j \in \cc0{t_k},\diam \left( \bigcup_{i=1}^{2^k} f^j(J_{k,i}) \right) <\eps
  \text{ and }\forall i,j\in\co1{2^k},f^{t_k}(J_{k,i})=f^{t_k}(J_{k,j}).\]
One can find an illustration of these sets for the tent map on Figure~\ref{fig:tent}.
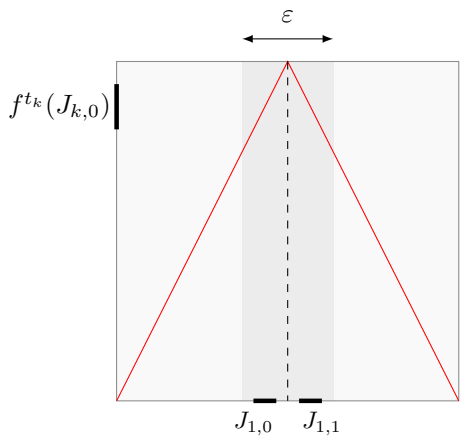
\begin{figure}
    \centering
    \begin{tikzpicture}[scale=0.3]
    %\begin{scope}[xshift=-17cm]
     %   \fill[color=gray!05] (0,7.5) rectangle (15,15);
    %fill[gray!15] (5,7.5) rectangle (10,15);
    %\draw[dashed] (7.5,7.5) -- (7.5,15);
    %\draw[color=gray!90] (0,7.5) -- (15,7.5);
    %\draw (0,0) rectangle (15,15);
    %\end{scope}
    
    \node[scale=0.9] at (6,-1) {$J_{1,0}$};
    \node[scale=0.9] at (9,-1) {$J_{1,1}$};
    
    \fill[color=gray!05] (0,0) rectangle (15,15);
    \fill[color=gray!15] (5.5,0) rectangle (9.5,15);
    
    %\fill[color=gray!15] (1.5,0) rectangle (2,15);
    
    %\draw (1,-0.75) rectangle (2.5,0.75);
    %\draw[dashed] (1.75,0) -- (1.75,0.75);
        \draw[color=gray!90] (0,0) rectangle (15,15);
    \draw[color=red] (0,0) -- (7.5,15) -- (15,0);
    \draw[dashed] (7.5,15) -- (7.5,0);
    \draw[line width=0.6mm] (0,12) -- (0,14);
    \draw[line width=0.6mm] (6,0) -- (7,0);
    \draw[line width=0.6mm] (8,0) -- (9,0);
    
    \draw[latex-latex] (5.5,16) -- (9.5,16);
    \node at (7.5,17) {$\eps$};
    \node at (-2.5,13) {$f^{t_k}(J_{k,0})$};
    \end{tikzpicture}
    \caption{Illustration of the intervals $J_{k,i}$ and $f^{t_k}(J_{k,0})$ from the proof of Lemma~\ref{l:mix}, where $f$ is the tent map and $k=1$.}
    \label{fig:tent}
\end{figure}
The above statement implies that $f$ is not positively $(2^k-1)$-expansive: indeed, 
$J=f^{-t_k}(x)\cap\bigcup_{i=1}^{2^k} f^j(J_{k,i})$, where $x\in f^{t_k}(J_{k,0})$ is arbitrary, satisfies the hypothesis of Remark~\ref{r:fnposexp}.
%\item \textbf{Proof by induction.}
\begin{enumerate}%[(i)]
 \item \textbf{Base case.} The case $k=0$ is trivial{, for any choice of interval $J_{0,1}$ with length less than $\eps$}.
%	There exist $c^{-},c^{--} \in (c-\eps/2,c)$ and $c^{+}, c^{++} \in (c,c+\eps/2)$ such that, setting $J_{0,1}\defeq [c^{--},c^{-}]$ and $J_{0,1}\defeq [c^{+},c^{++}]$, $f(J_{1,0})=f(J_{1,1})$ is a closed non-degenerate interval of diameter strictly smaller than $\eps$.
%        By setting $t_1\defeq 1$, we have thus proved the case $n=0$.
\item	\textbf{Induction step.} Let us assume that the assumption has been proved for some $k\in\N$.
  Since $f$ is mixing and $c\in\Int I$, \cite[Proposition~2.8]{Ruette} gives some $t>0$ such that 
  \[c \in\Int{f^{t}(f^{t_{k}}(J_{k,0}))}.\]
% Indeed, considering an open interval on the left of $c$ and another one on the right of $c$, by mixing there exists 
% an integer $T$ such that for all $t \ge T$, $\Int{f^{t}(f^{t_{k}}(J_{k,0}))}$ has non-empty intersection with both of these open 
% intervals. As a consequence, by the intermediate value theorem, $\Int{f^{t}(f^{t_{k}}(J_{k,0}))}$ contains $c$.
%	\comm{Pierre}{aren't we here assuming some kind of topological exactness: any open set reaches any specific point? maybe $c$ is simply very very close to this set}
%	\comm{Piotr}{Unfortunatelly, there are mixing but not exact maps}
% By the mean value theorem,
  {
    By uniform continuity of $f^j$, there exists $\delta>0$ such that for every subinterval $J\subset I$ with length less than $\delta$ and every $j\in\co0t$, $\diam f^j(J)<\eps$.
    By Lemma~\ref{lem.interm} applied to $f^t$, there exists a subinterval $J \subset f^{t_{k}}(J_{k,0})$ with length less than $\delta$ (hence for every $j \in \cc0t$,  $\diam f^j(J)<\eps$) such that
	\[c\in  %\overset{\circ}
	\Int { f^{t}(J)}.\]

      Set $t_{k+1}=t_k+t$.
      Since $c$ is a local extremum, there are disjoint intervals $I^{-}\subset f^{t}(J)\cap[0,c)$ and $I^{+}\subset f^{t}(J)\cap(c,0]$ such that $f(I^{-})=f(I^{+})$.
      For every $j\in\co1{2^k}$, the induction step gives that $J_{k,j}\cap f^{-t-t_k}(I^-)$ and $J_{k_j}\cap f^{-t-t_k}(I^+)$ are nontrivial intervals.
      Hence one can apply Remark~\ref{r:interval.preimage} to the restriction of $f^{t_{k+1}}$ to these subintervals and to $J$: one finds two subintervals $J_{k+1,2j}\subset J_{k,j}\cap f^{-t_{k+1}}(I^-)$ and $J_{k+1,2j+1}\subset J_{k_j}\cap f^{-t_{k+1}}(I^+)$ such that $f^{t_{k+1}}(J_{k+1,2j})=f(I^{-})=f(I^{+})=f^{t_{k+1}}(J_{k+1,2j+1})$.
      The collection $(J_{k+1,i})_{i}$ satisfies the wanted properties.
      }
\popQED\end{enumerate}
%\end{enumerate}
\end{proof}

%\pierre{
\begin{theorem}\label{theorem:interval}
  No interval system is finitely positively expansive.
\end{theorem}
\begin{proof}
  By Proposition~\ref{p:sensmix} and Item~\ref{i:sens} of Corollary~\ref{c:equicontinuous}, any finitely positively expansive interval system would admit a non-degenerate subinterval $J$ and $p>0$ such that $f^p\restr J$ is mixing.
  From Remark~\ref{remark:elem}, $f^p\restr J$ would also be a finitely positively expansive system of an interval.
  This contradicts %Theorem~\ref{int:mix}.
Lemma~\ref{l:mix}.
\end{proof}
%}

\subsection{Hawaiian earring}
%\pierre{
 % These remarks from previous versions of our article would give the first point of the Hawaiian theorem. 
 
%We start this section by simple observations on dynamics on topological graphs formed by bouquet of circles.
 
We present exploratory result towards Question~\ref{question:continua}.
What differentiates the interval and the circle regarding finite positive expansiveness is the possibility to have locally invertible expanding continuous maps. 
The example of \textbf{Hawaiian earring} reveals another (topological) property which prevents 
existence of finitely positively expansive maps. 
The Hawaiian earring is defined as $\mathcal{H} \defeq \bigcup_{n \in\Z_+} C_n$, where:

\[C_n\defeq\sett[\R^2]{(x,y)}{\left(x - \frac{1}{n}\right)^2 + y^2 \le \frac{1}{n^2}}.\]
This definition is illustrated on Figure~\ref{figure:hawaii}.

\begin{figure}[!ht]
\begin{center}
\begin{tikzpicture}[scale=1]
\draw (0,0) circle (40pt);
\draw (-0.7,0) circle (20pt);
\draw (-1.05,0) circle (10pt);
\draw (-1.225,0) circle (5pt);
\node at (.5,0) {$\ldots$};
\end{tikzpicture}
\end{center}
\caption{Illustration of the Hawaiian earring.\label{figure:hawaii}}
\end{figure}
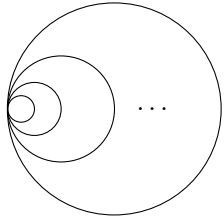

The Hawaiian earring is an example of topological graph.
The formal definition of topological graph can be found, for instance, in \cite{topgrph}.
What is relevant here is that a \textbf{topological graph} $G$ is a locally connected union of (possibly countably many) topological intervals (i.e. subsets of $\mathbb{R}^3$ homeomorphic to the interval $I$) which do not intersect in their interiors (i.e. $e\cap\Int{e'}\ne\emptyset$ for every distinct such intervals $e,e'$). %such that for every $e \in E$ there exists a homeomorphism $f_e : [0,1] \rightarrow e$ such that for all $e,e'$ distinct, $f_e (]0,1[) \cap f_{e'} (]0,1[) = \emptyset$.
The \textbf{degree} of a point $w\in G$, denoted by $\deg_{G}(w)$, is the limsup of the number of connected components of $\ball[\eps]w\setminus\{w\}$, when $\eps$ goes to $0$.
%A \textbf{topological graph} $G$ is the union of (possibly countably many) topological intervals (i.e. spaces homeomorphic to the interval $I$) which do not intersect in their interiors, i.e. $e\cap\int(e')\ne\emptyset$ for all topological intervals $e,e'$. % of a set $E$ of subsets of $\mathbb{R}^2$ such that for every $e \in E$ there exists a homeomorphism $f_e : [0,1] \rightarrow e$ such that for all $e,e'$ distinct, $f_e (]0,1[) \cap f_{e'} (]0,1[) = \emptyset$.
%The \textbf{degree} of a point $w\in G$, denoted by $\deg_{G}(w)$, is the supremum of the number of connected components of $W\setminus\{w\}$, where $W$ ranges among the neigbhorhoods of $w$.
By definition, the degree of any point is in $\cc1\infty$.
For instance, every point in the interior of one interval from the definition of the graph has degree $2$. %for all $e \in E$, every point in $f_e(]0,1[)$ has degree at least 2.
We call \textbf{vertex} every point which has degree $1$ or at least $3$.
A \textbf{subgraph} of $G$ is simply a closed subset $G' \subset G$ (in that case it is also a topological graph).

\begin{lemma}\label{lem:deg}
  Let $G$ be a topological graph, $f:G\to G$ continuous such that $(G,f)$ is a positively $n$-expansive system, and $G'\subset G$ a subgraph.
  Then for every $w\in G'$, $\deg_{G'}(w)\le n\deg_{f(G')}(f(w))$.
\end{lemma}
\begin{proof}
  Were it not the case, $f$ would map strictly more than $n$ disjoint non-trivial (open) intervals with endpoint $w$ onto some common non-trivial interval with endpoint $f(w)$.
  Therefore, one could find a set of $(n+1)$ points arbitrarily close to each other (and to $w$), all mapped to the same point, contradicting Lemma~\ref{l:locally.preimages}.
\end{proof}

% A \textbf{bouquet of circles} is any topological graph obtained from a finite union of copies of the circle $\mathbb{T}_1$, by identification of the copies of the point $0 \in \mathbb{T}_1$.
% For every finite subset $I$ of $\Z_+$, we denote by $C_I$  the union of the circles $C_n$, $n \in I$. Notice that when $I$ is finite, $C_I$ is a bouquet.
% The following remarks provide an idea of what could be said in this direction:
\begin{remark}
For any integer $n$ and every $k \le n$, we can construct a map $f$ on $C_{\cc1n}$ which is positively $k$-expansive and not positively $(k-1)$-expansive: for instance a map which fixes $0$, acts as $z \mapsto z^{2}$ on each circle $C_k , \cdots , C_n$, and maps the circles $C_1, \cdots , C_{k-1}$ onto $C_k$ in such a way that for every $i<k$,  $f\restr{C_i}$ is one-to-one.
\end{remark}

%\begin{remark}
%	If a topological graph admits a positively $n$-expansive map $f$, for all vertex $w$ in this graph, the degree of $w$ is bounded 
%	by $n$ times the degree of $f(w)$. Indeed, 
%	otherwise there would be strictly more than $n$ non-trivial intervals whose endpoint is $w$ which are mapped onto a non-trivial interval whose endpoint is $f(w)$. Therefore one can find a set of $(n+1)$ points arbitrarily close to each other and arbitrarily close to $f(w)$
%	which are mapped onto the same point. This enters in contradiction with positive $n$-expansiveness.
%	%the degree of $f(w)$ is bounded by $n$ times the degree of $w$.
%\end{remark}
%%}

%On the other hand, infinite number of circles is obstruction of expansivity.

\begin{proposition}
There is no finitely positively expansive map on the Hawaiian earring.
\end{proposition}
%\comm{Pierre}{exprimer cela avec un facteur via $\kappa$}
\begin{proof}
  Let us consider a map $f$ from the Hawaiian earring to itself and assume ad absurdum that it is finitely positively expansive, and fix $\eps>0$ and $N$ such that for all $x$, $\card{W_{\eps}(x)}\le N$.
\begin{enumerate}
\item \textbf{The accumulation point $\zero$ is fixed by $\boldsymbol{f}$.} % Indeed, let us assume ad absurdum that $f(\zero) \neq \zero$. For all $n$, let us denote by $I_n$ the open interval of length $1/2$ centered on $\zero$ in $C_n$. If there is some $n$ such that $f(I_n)$ is reduced to a point, then $f^{-1}(f(\zero))$ is infinite, 
% which is impossible by Corollary~\ref{cor.bounded.fibers}. Thus for all $n$, $f(I_n)$ is not reduced to a point. This implies 
% that there is a sequence of integers $(n_k)_{k \in\N}$ such that for all $k$, the intersection of the sets $f(I_{n_l})$, $l \le k$, contains a point different from $f(\zero) $. It has one preimage in each $C_{n_l}$, $l \le k$, and none of these preimages can be equal to $\zero$, since it is different from $f(\zero) $. As a consequence this point has at least $k$ preimages by $f$. Since there is such a point for all $k \in\N$, there is contradiction with Corollary~\ref{cor.bounded.fibers}, and we conclude that indeed $f(\zero) = \zero$ (cf. Remark~\ref{r:deg}).
This is a direct application of Remark~\ref{lem:deg} to $w=\zero$, because $\deg_{\mathcal H}(\zero)=\infty$.
\item \textbf{There exists a finite $\Lambda\subset \Z_+$ and some map $\kappa:{\Z_+} \backslash \Lambda\to{\Z_+}$ such that for every $n\in {\Z_+} \backslash \Lambda$, $f(C_n)$ is contained in $C_{\kappa(n)}$.}
  Indeed, let $\Lambda\defeq\sett[{\Z_+}]n{f^{-1}(\zero)\cap C_n\setminus\{\zero\}}\ne\emptyset$.
  By Corollary~\ref{cor.bounded.fibers} and disjointness of all $C_n\setminus\{\zero\}$, $\Lambda$ is finite.
  Moreover, for $n\in{\Z_+}\setminus\Lambda$, the image of $C_n\setminus\{\zero\}$ by $f$ is a connected set not containing $\zero$, hence it is contained in some circle $C_{\kappa(n)}$ (and so is $C_n$, since $f(\zero)=\zero\in C_{\kappa(n)}$).
% Since $\zero$ cannot have infinitely many pre-images, and as a consequence of point (1), there exists a set $\Lambda \subset\Z_+s$ such that $\Lambda^c$ is finite, and for every $n \in \Lambda$, the image of $C_n$ is %equal to 
% contained in some $C_{\kappa(n)}$.
%\comm{Pierre}{If i understand correctly there are two statements.
  %``0 is mapped to 0'' needs further proof.
  %Then the existence of $\kappa$ seems to use this (change references of ``point'' to some real ref).
  
% Maybe simply:
  % By Corollary~\ref{cor.bounded.fibers}, $f$ is bounded-to-one, so that the infinitely many disjoint open paths $]x_n,\zero[$ must be mapped to infinitely many paths $]f(x_n),f(\zero)[$.
%}

%Then there exists an open set $u$ containing $\zero$ whose image 
%is homeomorphic to an open interval and such that $\diam(u) < \eps$.
%On the other hand, for every $n \ge 1$, 
%$f(C_n)$ is homeomorphic to a compact 
%interval. Furthermore, $f(C_n)$ cannot be reduced to a point, otherwise $f$ would not be finitely positively expansive. Therefore $f(C_n)$ has 
%positive length. Since for every $n$, 
%$f(C_n)$ contains $f(\zero)$,
%$f(\zero)$ has a pre-image in 
%every $C_n$. Since there is an infinity of integers $n$ such that $C_n$ is contained in $u$, $W_{\eps}(\zero)$ is infinite. 
%Since this is true for every $\eps$, 
%there is a contradiction.
\item \textbf{There exists a finite $\Gamma\supset\Lambda$ such that for every $n\in{\Z_+}$, there exists $t\in{\N}$ such that $\kappa^t(n)\in\Gamma$.
    % $\bigcap_m \kappa^{-m} \left(\Lambda \cap \llbracket p, \infty \llbracket \right)$ is not empty.
  }
  % \comm{Pierre}{
  Indeed, by Point~\ref{i:sens} of Corollary~\ref{c:equicontinuous}, $f$ is $\eps$-sensitive.
  There exists $p$ such that $\diam(C_{\co p\infty})\le\eps$.
  Let $\Gamma\defeq\co0p\cup\Lambda$.
  Let $n\in{\Z_+}$, $x\in C_n\setminus\{\zero\}$, and $\delta>0$ such that $\ball[\delta]x\subset C_n$.
  By $\eps$-sensitivity, there exists $t\in{\N}$ and $y\in\ball[\delta]x$ such that $d(f^t(x),f^t(y))>\eps$.
  This means that at least one of the points $f^t(x)$ and $f^t(y)$ is in $C_{\co0p}\subset C_{\Gamma}$.
  Since $x,y\in C_n$, we deduce that either $\kappa^t(n)\in\Gamma$, or $\kappa^t(n)$ is undefined.
  In the latter case, it means that there exists $t'<t$ such that $\kappa^{t'}(n)\in\Lambda\subset\Gamma$.
% %    Then, one can consider the $f$-invariant bouquet $B\defeq\bigcup_{n<m}\bigcup_{t\le t_{\kappa(n)}}C_{\kappa^t(n)}$, which is the limit set.
% %  }
%  Since these sets are all stable under the action of $\kappa$, there are two possible cases: 
% $\kappa$ has an infinite orbit or infinitely many finite orbits. In both cases, there is some $C_k$ in the ball centered on $\zero$
% of diameter $\eps$ whose trajectory under the action of $f$ stays in this ball.
% This means that $f$ is not finitely positively expansive (not even sensitive), which is a contradiction.
% %{\blue From this point, else: $\kappa$ has infinitely many finite orbits or for every $N$ there exists some $n>N$ such that $\kappa^k(n)>\kappa(n)$ for every $k \ge 1$.}
%\item \textbf{When there exists some $p \in\N$ such that $\bigcap_m \kappa^{-m} \left(\Lambda \cap \llbracket p, \infty \llbracket\right)$ is empty.}
%\begin{enumerate}
\item \textbf{There is a finite $J\supset\Gamma$ such that $f(C_J) \subset C_J$, and for every $n\in{\Z_+}$, there exists $t\in{ \N}$ such that $f^t(C_n)\subset C_J$. %for every $n\notin J$,
  %\comm{Pierre}{a bit fast; explicit $B$ please? Is it $C_{\Lambda^C\cup\cc0n}$? say why it's invariant (so far we didn't say much about $\Lambda^C$).
    %A possibility could be to consider the limit set $\Omega_f$, which is the maximal surjective subsystem.
  }
%    there exists some $k_n$ such that $f^{k_n}(C_n) \subset B$.}
    Let us prove this. First, for every $n\in{ \Z_+}$, the set $f(C_n)$ is contained in some $C_{D_n}$ for some finite set $D_n\subset\Z_+$ (otherwise the point $\zero$ would have infinitely many preimages, which is not possible).
      Consequently, $f(C_\Gamma)$ is contained in some $C_D$ for some finite set $D\defeq\bigcup_{n\in\Gamma}D_n\subset\Z_+$. %\defeq\sett[{\blue \Z_+}]n{f(C_{\Gamma})\cap C_n\ne \{0\}}$ is finite. %{\red It is nonstandard notation!}
%    and $I\defeq D\setminus\Gamma$.
%the finite subset of $\Z_+$ such that $D \backslash C_{\Gamma} \cup \{\zero\} = C_I$.
 Since $D\setminus\Gamma$ does not intersect $\Lambda$, $\kappa$ is well defined over $D\setminus\Gamma$.
 Furthermore, by the previous point, for every $n \in D\setminus\Gamma$, %\cup  \left( \Lambda \cap \llbracket 0, p \llbracket \right)$
 there exists some $t_n > 0$ such that $\kappa^{t_n}(n) \in \Gamma$. %Lambda^c \cup \left( \Lambda \cap \llbracket 0, p \llbracket \right)$. We then can take $B=C_J$, where $J$ is as follows:
 Now let
\[J\defeq\Gamma \cup \sett{\kappa^{t'}(n)}{n\in D\setminus\Gamma,0\le t'\le t_n}.\] % \cup \left( \Lambda \cap \llbracket 0, p \llbracket \right), l \le m_k\}.\]
Let us prove that $f(C_J) \subset C_J$.
If $m\in\Gamma$, then $f(C_m)\subset C_D\subset C_{\kappa^0(D\setminus\Gamma) \cup \Gamma}\subset C_J$.
If $m\in J\setminus\Gamma$, then there exists $n\in D\setminus\Gamma$ and $t'<t_n$ (note that $t'=t_n$ implies that $m\in\Gamma$) such that $m=\kappa^{t'}(n)$; then $f(C_m)\subset C_{\kappa(m)}=C_{\kappa^{t'+1}(n)}\subset C_J$.
The other statement is directly derived from the previous point and the fact that $J\supset\Gamma$.
% Furthermore, the fact that for every $n$ such that $C_n$ is not contained in $B$, there exists some $k_n$ such that $f^{k_n}(C_n) \subset B$ comes directly from the hypothesis that $\bigcap_m \kappa^{-m} \left(\Lambda \cap \llbracket p, \infty \llbracket\right)$ is empty.}
% \item \textbf{There exists an integer $d$, some integer $t$ and a set $K \subset \Z_+ \backslash J$ such that for all $k \in K$, 
% there exists an interval $V_k$ of positive length in $C_k$ which contains $\zero$ such that $f^t(V_k) \subset C_d$ and $|K| \ge 2(N+1)$. } In order to see this, it is sufficient to see that there exists some $t\in\N$ and some set $K' \subset\Z_+\setminus J$ such that %for all $k \in K$,
%     $f^t(C_{K'}) \subset C_J$ and $|K'| \ge  2(N+1)|J|$. This is straightforward from the last point. Then we apply the pigeon-hole 
% principle to obtain that there is a subset $K \subset K'$ and some $d \in J$ such that for all $k \in K$, 
% there exists an interval $V_k$ of positive length in $C_k$ which contains $\zero$ such that $f^t(V_k) \subset C_d$ and $|K| \ge 2(N+1)$. 
\item  {\bf The system is not positively $N$-expansive.} %with constant $\eps$.}
  Let $n>N\card J$.
  From the previous point, there exists $t\in{ \N}$ such that $f^t(C_{\cc1n}) \subset C_J$.
  Then $\deg_{C_{\cc1n}}(\zero)=n>N\card J\ge N\deg_{f^t(C_{\cc1n})}(f^t(\zero))$, which, by Lemma~\ref{lem:deg}, gives that $f^t$ cannot be positively $N$-expansive.
  We conclude by Remark~\ref{remark:elem}.
 %      One can assume that for all $k \in K$, $l \le t$, $\text{diam}(f^l (V_k)) <\eps$.
%       \pierre{The set $f^t(V_k)$ is connected and contains $\zero$ (because $\zero$ is fixed by $f$), not reduced to a point (because such a point would have uncountably many preimages), hence includes some non-trivial arc of $C_d$ on one of the two sides.
% %      Since $\zero$ is fixed by $f$, %it is also fixed by $f^t$, thus
% %      $\zero \in f^t (V_k)$ for all $k \in K'$.
% %      Since none of the sets $f^t(V_k)$, $k \in K'$, can be reduced to a point, t
%         There exists $L \subset K$ with $|L| \ge N+1$ such that the corresponding arcs are on the same sides, so that the intersection of the sets $f^t (V_k)$, $k \in L$, is also a non-trivial arc of $C_d$. This implies that there is a point $z$ different from $\zero$ in this intersection, and consequently, there is a set $F\subset f^{-t}(z)$ %contained in its preimage by $f^t$ 
%         such that $\diam{f^l(F)} < \eps$ for all $l \le t$ and $|F|=|L|>N$.
%         By Remark~\ref{r:fnposexp}, we get a contradiction with the hypothesis that the system $(\mathcal{H},f)$ is positively $N$-expansive with constant $\eps$.
%\comm{Pierre}{here also, please be a little more explicit. Then we can use Remark~\ref{r:fnposexp}.}
\popQED
%\end{enumerate}
%{\red It seems there is yet another case to be considered. Some circles go into cycle (or even the last one is fixed) and all other "grow" to it. I believe that then either we have an invariant interval, or infinitely many points joining into a trajectory (or in more general, an asymptotic orbit). But have to prove it in detail, since it in neither of two cases above.}
\end{enumerate}
\end{proof}

This proof can be generalized to graphs with one infinite-degree vertex whose all neighborhoods have complements which are graphs defined as finite unions of intervals.

%As a consequence of %Lemma~\ref{lemma.locally.preimages}:

%\begin{corollary}
%So in a graph $X$ and $f$ such that the system $(X,f)$ is positively $n$-expansive, if $v = f(w)$, then $deg(w) \le deg(v) * n$.
%\end{corollary}

%\begin{remark}
%The maximum of the function considered in the statement of Lemma~\ref{lemma.locally.preimages}
%is independent from the minimum $n$ such that the system is positively $n$-expansive. For instance one can consider $X$ equal to a finite set of $k$ points and $f$ which sends all these points to one of them. 
%\end{remark}

%\begin{remark}
%We can reduce this conjecture to the similar one in which no vertex has degree $1$. For this on every vertex $v$ connected to at least one vertex of degree $1$, 
%we order to corresponding segments and send each to the following, except for
%the last one, which is sent on any segment of a circle which contains $v$.
%\end{remark}

%Maybe we need to add combinatorial constraints to this statement. 

\section{Two-sided shifts}\label{section:shifts}%The case of $\Z$-shifts\label{section.symbolic}}

It is well known that ${ \N}$-shifts are all positively expansive. We will see in this section that the case of $\Z$-shifts is more complex. 
Furthermore the facts that they are all homeomorphisms, and that $\Z$-shifts with the shadowing property are exactly the well-known shifts of finite type make this class a good playground for exploring generalizations of Theorem~\ref{theorem.shadowing}: this is the object of study of Section~\ref{section.predecessor}.

We also found some examples of $\Z$-shifts which are finitely positively expansive. %[Section~\ref{section:examples}]. 
As these examples can be obtained using substitutions, we explored the relation between substitutions and this property, and provide some sufficient conditions on the substitutions for the corresponding $\Z$-shift to be finitely positively expansive [Section~\ref{section:sufficient}]. We still leave open the following questions: 

\begin{question}
Which $\Z$-shifts are finitely positively expansive? Which substitutive / S-adic ones are? 
\end{question}

\subsection{Preliminary facts}

A standard reference to symbolic dynamics is \cite{Marcus}. Below we recall a few standard facts that will be used later.
Let $\A$ be a finite set which we will call \textbf{alphabet}. The \emph{shift} action, denoted by $\sigma$,
is the function from $\A^{\Z}$ to itself 
such that for every $x$ in this set and $i \in \Z$, $\sigma(x)_i = x_{i+1}$.
A $\Z$-shift $X$ on alphabet $\A$ is a dynamical system formed by the shift action restricted to a compact subset of $\A^{\Z}$ which is stable under the shift action.
{We will denote by $\lang[r]X\defeq\sett{x_{\co0r}}{x\in X}$ its \textbf{language} of length $r\in{\N}$.

  %\begin{definition}
  %We say that a shift is periodic when all configurations are periodic, with a uniform period. It
  %is aperiodic when it does not contains any periodic point.
%\end{definition}
%Let us generalize this a bit further.
In the following, for every finite set $\A$ and every $z\in \A^{{\N}}$, we will denote by $\mirr z$ the element of $\A^{-\N}$ defined by: for all $k<0$, ${\mirr z}_k = z_{-k-1}$.
  Furthermore, for every $z' \in \A^{{ \N}}$ and $z \in \A^{{ \Z_-}}$, we will simply denote by $zz'$ the element of $\A^{\Z}$ which coincides with $z$ on ${ \Z_-}$, and with $z'$ on ${ \N}$.
  For every $a \in \A$, we will denote by $\uinf a$ the infinite word $z \in \A^{{ \N}}$ such that for every $k \in\N$, $z_k = a$.
  We also simplify the notation $\mirr{\uinf a}$ into $\pinf a$.

An asymptotically periodic configuration is one such that there exist $q \in \Z$ and $p \in {\Z_+}$, respectively called \textbf{preperiod} and \textbf{asymptotic period} such that for every $i\ge q$, $x_{i+p}=x_i$.

%   Let us sketch a fondamental result, which can be seen as a consequence of \cite{schwartzman} or \cite[Theorem~3.7]{boylelind}, or even of the proof of \cite[]{symbol}.
% \begin{proposition}\label{p:posexpfin}
%   A $\Z$-shift is positively expansive if and only if it is finite.
% \end{proposition}
% \begin{proof}
%   If $X$ is positively expansive, then compactness gives a diameter $\ell\in\N$ such that every $x,y\in X$ such that $x_{\co0\ell}=y_{\co0\ell}$ must have $x_{-1}=y_{-1}$.
% %  Up to conjugacy (considering the higher-$r$-block representation), one can consider without loss of generality that $\ell=1$.
%   In other words, there is a map $\phi:\lang[\ell]X\to\lang[\ell]X$ such that for every $x\in X$ and every $i\in\Z$, $x_{\co i{i+\ell}}=\phi(x_{\co{i+1}{i+1+\ell}})$.
%   By the pigeon-hole principle, one gets that there exist distinct $j,k\in\cc{i}{i+\card{\lang[\ell]X}}$ with $x_{\co j{j+\ell}}=x_{\co k{k+\ell}}$.
%   This period $k-j$ is inherited to the left: for every $i\le j$, $x_{\co i{i+\ell}}=\phi^{j-i}(x_{\co j{j+\ell}})=\phi^{j-i}(x_{\co k{k+\ell}})=x_{\co{i+k-j}{i+k-j+\ell}}$.
%   Since possible periods are within a finite set, there is actually a uniform period.
% \end{proof}
}

A shift $X$ is said to be of \textbf{finite type} when there exists a finite set of words $\mathcal{F}$ such that $X = \sett[\A^\Z]x{\forall i,j\in\Z,x_{\co ij}\notin \mathcal{F}}$.
The following is a well known fact by Walters \cite{Wal}:
\begin{proposition}
Every $\Z$-shift has the shadowing property if and only if it is of finite type. 
\end{proposition}

%As a consequence we can provide another proof of Theorem~\ref{theorem.shadowing} in the context of $\Z$-shifts using the following proposition:

As a consequence, Theorem~\ref{theorem.shadowing} immediately implies the following.

\begin{corollary}\label{c:finite.type}
Any finitely positively expansive $\Z$-shift of finite type is finite. 
\end{corollary}

It is natural to wonder if this result can be extended to sofic shifts, since they are to some extent similar to shifts of finite type.
They are usually defined as follows:
\begin{definition}
A $\Z$-shift $X$ is \textbf{sofic} when there exists a shift of finite type $Z$ and a %surjective homomorphism 
 factor map $\phi : Z \rightarrow X$ (i.e. a shift-commuting continuous surjective map).
\end{definition}

Using a characterization of sofic shifts in terms of predecessor sets, we can indeed generalize the result to sofic shifts.
As a matter of fact, we will see that {it even holds for} shifts with countably many predecessor sets.

\subsection{Predecessor sets}\label{section.predecessor}

There is a natural relation between sofic shifts and finite positive expansiveness, as they can be both characterized in terms of predecessor sets.

\begin{definition}
A \textbf{predecessor} of infinite word $z'\in\A^{{ \N}}$ in $\Z$-shift $X$ is some left-infinite word $z\in\A^{{ \Z_-}}$ such that $\exists x\in X,x_{{ \N}}=z'$ and $x_{{ \Z_-}}=z$.
%the \textbf{predecessor set} of $u$ is the following set denoted by $\mathcal{P}_X(u)$: 
%\[\mathcal{P}_X(u)\defeq \left\{(v_k)_{k <0} : { \exists w \in X, w_{| \llbracket - \infty , 0\llbracket} = v, w_{| \llbracket 0 ,  + \infty \llbracket} = {\red u}%v
%}\right\}.\]
\end{definition}

A proof of the following can be derived from~\cite[Section~3.2]{Marcus}:
\begin{theorem}\label{theorem.sofic.followers}
A $\Z$-shift is sofic if and only if it has a finite number of predecessor sets. 
\end{theorem}

The following is straightforward from the definition of predecessor sets:
\begin{lemma}\label{l:preds}
A $\Z$-shift is positively $n$-expansive if and only if all of its predecessor sets have at most $n$ elements. Equivalently, for every infinite word $z$ which is the restriction of a configuration of the shift on $\co0\infty$, and every integer $\ell$, the number of predecessors of $z$ with length $\ell$ is at most $n$.
\end{lemma}

\begin{corollary}\label{corollary.countable}
    Every finitely positively expansive $\Z$-shift with countably many predecessor sets is finite.
\end{corollary}
\begin{proof}
    It is sufficient to see that these conditions imply that the shift is countable, and to use Item~\ref{i:indnb} of Corollary~\ref{c:equicontinuous}.
\end{proof}

\begin{corollary}\label{cor:sofic_exp}
Every finitely positively expansive sofic $\Z$-shift is finite. 
\end{corollary}

\subsection{{S-adic shifts}}\label{section:sufficient}
In this section, we consider $\Z$-shifts which are obtained from substitutions. For a finite alphabet $\A$, let us denote by $\A^+$ the set of non-empty words over $\A$.
A \textit{substitution} is a function $\tau : \A \to \B^+$, where $\A$ and $\B$ are finite sets. In the following, all substitutions are assumed to be non-erasing, meaning none of its images is the empty word. Such a function can naturally be extended into a monoid homomorphism $\tau^{*} : \A^{*} \to \B^{*}$. On the other hand it is more difficult to extend naturally on $\A^{\Z}$.
A more natural object to define in this case is the set of configurations in $\B^{\Z}$ which can be obtained by application of the substitution on an element of $\A^{\Z}$. We say that such configurations can be \emph{desubstituted}.

\begin{definition}
%We will only consider injective ones.
Let us consider $\tau : \B \to \A^+$ a substitution.
For every $x\in\B^{\Z}$, 
we call \textbf{desubstitution scheme} of $x$ for $\tau$ any increasing sequence $\kk=(k_s)_{s\in\Z}$ in $\Z$ such that there exists $y\in\A$, such that for every $s\in\Z$ we have:
\[x_{\co{k_s}{k_{s+1}}} = \tau(y_s).\]
The sequence $y$ is called a \textbf{desubstitution} of $x$. The desubstitution scheme and desubstitution are said to be \textbf{standard} when 
$\max\sett[\Z]s{k_s \le 0}=0$.
\end{definition}

We will denote by $|w|$ the length of $w$, when $w$ is a word on some alphabet $\A$. For two substitutions $\tau : \B \rightarrow \A^{+}$ and $\tau' : \mathcal C \rightarrow \B^{+}$, 
the restriction of $\tau^{*} \circ {\tau'}^{*}$ to $\mathcal{C}$ is a substitution which is denoted, by abuse of notation, by $\tau\tau'$. Furthermore, for any substitution $\tau : \B \rightarrow \A$, we will set $\len\tau \defeq\min_{a\in\B}\len{\tau(a)}$
and  $||\tau|| \defeq\max_{a\in\B}\len{\tau(a)}$, and call these numbers respectively the \textbf{minimal length} and  \textbf{maximal length} of $\tau$.

\begin{remark}\label{r:rad}
    Provided a substitution $\tau$, if $\kk$ is a standard desubstitution scheme of $x$ for $\tau$, then for every $i\in{ \N}$, we have \[i\len\tau-||\tau||<i\len\tau+k_0\le k_i\le i||\tau||+k_0\le i||\tau||.\]
%    {A consequence is that for every right-quasi-recognizability radius $R$, the integer $j$ given by Definition~\ref{d:rad} satisfies}
%    %Definition~\ref{d:rad} thus gives
%   \[R\frac{\len\tau}{\norm{\tau}}\le j\le R\frac{\norm\tau}{\len\tau}.\]
%   Indeed, the second inequality derives from $j |\tau| \le k'_j$, $k'_j =k_R$ and $k_R \le R\norm{\tau}$, where $\kk$ and $\kk'$ are as in Definition~\ref{d:rad}. The first one is obtained symmetrically.
%    %Because of the symmetry in $x,x'$ of Definition~\ref{d:rad}, we have $\kk'_{\co R\infty}=\kk_{\co{2R-j}\infty}$, which is stronger when $j>R$.
%%{\red I see $R\len\tau\le k_R\le R\norm{\tau}$. How do we get above?}    
%% so that $i$ can be swapped with $j$.   Moreover, the cardinality condition implies that $%k_i\le R\norm\tau$, because there is a $k_n$ at least in each pattern of length $\norm\tau$.
\end{remark}

\begin{remark}\label{r:desubshift}
    Let us consider a substitution $\tau:\B\to\A^+$, $x\in\A^\Z$ and $i\in\Z$.
    Moreover, if $\kk$ is a desubstitution scheme of $x$ for $\tau$, and $j\in\Z$, then $\kk-j$ %(understood as difference at each coordinate) 
    is a desubstitution scheme for $\sigma^j(x)$, and if $\kk$ was standard, then $\sigma^m(\kk-j)$ is a standard desubstitution scheme for $\sigma^j(x)$, where $m\defeq\max_{k_i\le j}i$. As a direct consequence of Remark~\ref{r:rad}, $m\in\cc{\ipart{\frac j{\norm\tau}}}{\spart{\frac j{\len\tau}}}$.
\end{remark}

\begin{definition}
Let $\boldsymbol{\A} = (A_t)_{t\in{ \N}}$ be a sequence of non-empty finite sets and $\ttau=(\tau_t)_{t \in\N}$ be 
a sequence of substitutions such that 
for every $t \in\N$, $\tau_t$ is a function from $\A_{t+1}$ to $\A_t^+$. Such a sequence is called \textbf{directive sequence} on $\boldsymbol{\A}$.
%\textbf{directive sequence}, where each $\tau_t:\A_{t+1}\to\A_t^+$ is a substitution.
%Note that, for every $t\in\N$, $\tau_{\co0t}=\tau_0\tau_1\cdots\tau_{t-1}:\A_t\to \A_0^*$ is also a substitution.
\end{definition}

The number $\liminf_{t\to\infty}\card{\A_t}$ (which can be positive or infinite) is called the (alphabet) \textbf{rank} of the directive sequence $\ttau$. %, denoted by $r(\ttau)$
%We denote $\Ns\defeq\N\setminus\{0\}$.
We also set, for every $t',t\in\N$ such that $t\le t'$, 
$\tau_{\co t{t'}} \defeq \tau_{t} \cdots \tau_{t'-1}$, and $\sigma^{t}\ttau\defeq(\tau_{t'})_{t'\ge t}$ the shifted directive sequence.
For every $t \in\N$, we define $\tau_{\co0t}(\A_t^\Z)$ as the set of elements of $\A_0^\Z$ %{\red  $\A_t^\Z$?}
which admit a desubstitution for 
$\tau_{\co0t}$.
It is straightforward to see that for every $t \in\N$, an element of $\A_0^\Z$ admits a desubstitution for $\tau_{\llbracket 0 , t+1 \llbracket}$ if and only if it admits a desubstitution for $\tau_{\co0t}$ which itself admits a desubstitution for $\tau_t$. 

%{\red if it admits a desubstitution for $\tau_t$ which then admits a desubstitution for $\tau_{\llbracket 0 , t \llbracket}$?}

\begin{definition}
The \textbf{S-adic limit set} associated to a directive sequence is the set defined as follows:
\[\Omega_{\ttau} \defeq\bigcap_{t\in\N} \tau_{\co0t}(\A_t^\Z).\]

%\[\Omega_{\ttau} \defeq\bigcap_{t\in\N}\bigcup_{s\in\Z}\sigma^s\tau_{\co0t}^{-1}(\A_t^\Z).\]
By compactness, $\Omega_{\ttau}$ is a $\Z$-shift.
\end{definition}
%$\Omega_\ttau$ is the set of configurations $y^0$ that admit a desubstitution $(y^t)$ for any $\tau_{\co0t}$ (one could equivalently require that $y^{t+1}$ is a desubstitution of $y^t$ by $\tau_t$).

\begin{remark}
The term \emph{S-adic shift} associated with a directive sequence $\boldsymbol{\tau}$ usually refers to the set of configurations whose patterns appear in at least some $\tau_{\co0t}(a)$, $t \in\N$, $a\in \A_t$.
This set is contained in $\Omega_{\ttau}$. Hence whenever the shift is finitely positively expansive on $\Omega_{\ttau}$, it is also on this set.
\end{remark}

{
  \begin{remark}\label{r:rk1}
    If there exists $t\in{ \N}$ such that $\A_t$ is a singleton $\{a\}$ (in particular if the rank is $1$), then $\Omega_\ttau$ is the shift orbit of the periodic configuration $\dinf{\tau_{\co0t}(a)}$.
    Alphabet rank $1$ is thus more constrained that measure-theoretical rank $1$, or even topological rank $1$.
\end{remark}

When all the substitutions in a directive sequence $\ttau$ are equal to the same substitution $\tau$, $\Omega_\tau\defeq\Omega_{\ttau}$ is then the limit set of substitution $\tau$.
The rank of the directive sequence is equal to the size of the alphabet $\A_0$ (hence is finite).

S-adic shifts can be much richer than substitution ones, since they include (in some weak sense) all minimal shifts (see \cite{sadicmin}).
%\end{remark}
In particular, they cannot all be finitely positively expansive, since some of them have positive entropy.
}
In the present section, we study sufficient conditions for S-adic limit sets (and thus the corresponding S-adic shifts) to have finite positive expansiveness.

{
\begin{definition}
We say that a directive sequence $\ttau$ is \textbf{everywhere-growing} when it satisfies the equality  $\lim_{t\to\infty}\len{\tau_{\co0t}}=\infty$.
 We say that it is \textbf{expanding} when each of the substitutions $\tau_t$, $t\in\N$, is \textbf{expanding}, which means that $\len{\tau_t}\ge2$.
\end{definition}

\begin{remark}\label{r:subad}
    A directive sequence $\ttau$ is everywhere-growing if and only if for every $t'\in{ \N}$, $\lim_{t\to\infty}\len{\tau_{\co{t'}t}}=\infty$.
    This is due to subadditivity: $\len{\tau_{\co0t}}\le\len{\tau_{\co0{t'}}}\len{\tau_{\co{t'}t}}$.
\end{remark}

The following is classical and straightforward:
\begin{remark}\label{r:teles}
A directive sequence is everywhere-growing if and only if it has a telescoping which is expanding, meaning that there exists an increasing sequence $(t_i)_{i \in\N}$ of non-negative integers such that $t_0 = 0$ 
and for every $i \in\N$, such that the directive sequence $\left(\tau_{\co{t_i}{t_{i+1}}}\right)_{i \in\N}$ is expanding.
\end{remark}
%This is an equivalence up to telescoping (that is, presenting the directive sequence with some compositions $(\tau_{\co{t_i}{t_{i+1}}})_{i\in\N}$).
%Equivalently, $\prod$

\subsubsection{Recognizability}
A substitution $\tau: \B \rightarrow \A^+$ is called \textbf{injective} if it is one-to-one as a map (not necessarily its extension). For configurations on $\Z$, the intuition behind injectivity is formalised as recognizability.

The following definition comes from \cite{bustos}, and is inspired by more classical definitions \cite{mosse,BST}.

{\begin{definition}
    A substitution $\tau : \B \to \A^+$ is said to be \textbf{quasi-recognizable} {over $X\subset\A^{\Z}$ if every $x\in X$ admits at most one} standard desubstitution scheme. %, which is then denoted $\kk_{\tau}(x)$. When the context is not ambiguous, we may simplify this notation into $\kk(x)$. %We say that $\tau$ is \textbf{quasi-recognizable} when it is quasi-recognizable on all $x \in \Omega_{\tau}$.
    It is \textbf{recognizable} over $X$ if every $x\in X$ admits at most one desubstitution.
  When the set $X$ is omitted, it is assumed to be the full shift (or equivalently, the set of desubstitutable configurations).
  We will say that a directive sequence $\ttau$ is \textbf{quasi-recognizable} if every $\tau_{\co0t}$ is quasi-recognizable over $\Omega_\ttau$ (in particular if every $\tau_t$ is quasi-recognizable over $\Omega_{\sigma^t\ttau}$). %, and similarly for \textbf{quasi-recognizable}.
\end{definition}

\begin{definition}\label{d:rad}
A %(higher-generation {\red why?})
\textbf{right-quasi-recognizability radius} of a substitution $\tau : \B \to \A^+$ over a set $Z\subset(\A^{\Z})^2$ %{ on which it is quasi-recognizable}
is a natural number $R\in{ \N}$ such that for every pair $(x,x')\in Z$ with standard desubstitution schemes $\kk$ and $\kk'${ and $x_{{ \N}}=x'_{{ \N}}$, there exist $j,j'\in\cc0R$ such that $\kk_{\co j\infty}=\kk'_{\co{j'}\infty}$ and $x_{\co{k_{j}}\infty}=x'_{\co{k'_{j'}}\infty}$.}
%{\red if $k_j=k_{j'}$ and $k_j\geq 0$ then 
%$x_{\co{k_{j}}\infty}=x'_{\co{k'_{j'}}\infty}$ is consequence of $x_\N=x'_\N$. Do we stat it that way, because we allow $k_0<0$?
%Pierre: yes, it's the purpose of the next remark.
%I agree it's not completely satisfying, if you have a better idea.}
We will say that $R$ is a right-quasi-recognizability radius over $X \subset \A^\Z$ when it is for $Z=X\times X$.
% A (higher-generation) \textbf{right-quasirecognizability radius} for substitution $\tau$ over the set $X$ of configurations is $R\in\N$ such that for every $x,x'\in X$ with standard desubstitution schemes $\kk$ and $\kk'$, if $x_{\co{k_{-R}}\infty}=x'_{\co{k_{-R}}\infty}$, then $\kk_\N=\kk'_\N$ and $x_{\N}=x'_{\co{k'_R}\infty}$.
\end{definition}
}
Let us notice that if $R$ is a right-quasi-recognizability radius over some set $Z$, then any $R' > R$ is also a quasi-recognizability radius over $Z$.

%\begin{remark}
Note that the last condition $x_{\co{k_{{j}}}\infty}=x'_{\co{k'_{j'}}\infty}$ derives straightforwardly from {$x_{{ \N}}=x'_{{ \N}}$,}
%the hypothesis 
except when $j=j'=0$ (because then $k_j$ may be negative).
Informally, {$R$ is a right-quasi-recognizability radius} if one can desubstitute uniquely any right-half configuration, up to $R$ {`pre-image letters'}. %on the right part of $\Z_{-}$}.
% up to R letter images on the left. 
%\end{remark}
The following remark will be useful in short time.
% \begin{lemma}\label{l:belong}
%    $R\in\Ns$ is a right-quasi-recognizability radius for $\tau$ over $Z$ if and only if for every pair $(x,x')\in Z$ with standard desubstitution schemes $\kk$ and $\kk'$ and $x_\N=x'_\N$, and every $i\ge R$, there exists $m\in\Z$ such that $\k_i=k_m$. %k_{\co R\infty}=\kk'_{\co m\infty}$ and $x_{\co{k_R}\infty}=x'_{\co{k_m}\infty}$.
% \end{lemma}
% \begin{proof}
%   Consider such $x,x'$ and $\kk,\kk'$.
%   If $R$ is a right-quasi-recognizability radius, then there exist $j,j'\in\co0R$ such that $\kk_{\co j\infty}=\kk'_{\co{j'}\infty}$, which clearly implies our condition. %and $x_{\co{k_j}\infty}=x'_{\co{k'_{j'}}\infty}$.
% %   Then it is clear that, in particular, $\kk_{\co R\infty}=\kk'_{\co{R-j+j'}\infty}$ and $x_{\co{k_R}\infty}=x'_{\co{k'_{R-j+j'}}\infty}$.

%   Conversely, assume that for every such $x,x'$ and $\kk,\kk'$, $\sett{k_i}{i\ge R}\subset\sett{k'_i}{i\in\Z}$.
%   By symmetry of the condition, $\sett{k'_i}{i\ge R}\subset\sett{k_i}{i\in\Z}$.
%   Let $m,m'\in\Z$ be such that $k_R=k'_{m'}$ and $k_m=k'_R$.
%   Since $R\ge1$, $m>0$ because $k_m=k'_R>0$, and $m'>0$ because $k_{m'}=k_R>0$.
%   Moreover, $m>R$ if and only if $k'_{m'}=k_R<k_m=k'_R$, if and only if $m'<R$.
%   % there exists $j\in\Z$ such that $\kk_{\co R\infty}=\kk'_{\co{m}\infty}$ and $x_{\co{k_R}\infty}=x'_{\co{k'_{m}}\infty}$, for some $m\in\Z$.
% %   If $m\le0$, then $k_R=k'_m\le0$
% \end{proof}
    
\begin{remark}\label{r:rmin}
    Let us consider fixed some right-quasi-recognizability radius $R$ on $Z$ for $\tau$.
    Consider some $z\in\A^{ \N}$ such that there is $x\in Z$ with $x_{ \N}=z$, and $x$ and $\kk$ be such that $k_R$ is minimal among all choices of such $x$ and their {standard} desubstitution schemes $\kk$ (this exists since $k_R\ge R|\tau|-||\tau||$, from Remark~\ref{r:rad}).
     Then for every $x'\in Z$ and every desubstitution scheme $\kk'$, there exists $j'\in\cc0R$ such that $\kk_{\co R\infty}=\kk'_{\co{j'}\infty}$ and $x_{\co{k_R}\infty}=x'_{\co{k'_{j'}}\infty}$.
%     The minimal possible $k_R$ among all $x\in Z$ satisfies that for every $x'\in Z$, $k'_j=k_R$ for some $j\ge R$.
%     We get that the minimal possible radius over $\{x\}\times X$ satisfies the definition, with $j\le R$.
    Besides, Remark~\ref{r:rad} again gives:
    \[j'\ge\frac{k'_{j'}}{||\tau||}=\frac{k_R}{||\tau||}>R\frac{\len\tau}{||\tau||}-1,\]
    so that in the end $j'\in\cc{\spart{R\frac{\len\tau}{||\tau||}}}R$.
    Note that if $R(1-\frac{\len\tau}{\norm\tau})<1$, then $j'=R$ (this happens when $R\in\{0,1\}$ for instance).
\end{remark}

\begin{proposition}~\label{p:recog}
A substitution $\tau$ is quasi-recognizable over $X$ if and only if it admits a right-quasi-recognizability radius over $X$.
\end{proposition}
%This is in particular the case if $\tau$ is quasi-recognizable {over the full shift}.

\begin{proof}
\begin{itemize}
\item[$\boldsymbol{(\Leftarrow)}$]
{Roughly speaking, the strategy for proving this is, for every $x \in X$, to use the definition of right-quasi-recognizability radius on elements of $\mathcal{O}(x)$. Let us get more precise.}
Let us assume that $\tau$ admits a right-quasi-recognizability radius $R$ over $X$, and consider $z \in X$ which can be desubstituted, and $\kk,\kk'$ two standard desubstitution schemes for $z$.
Let us fix $i\in\co R\infty$, and set $y\defeq\sigma^{k_{-i}}(x)$, which admits both $(k_{n-i}-k_{-i})_{n \in \Z}$ and $(k'_{n-i'}-k_{-i})_{n \in \Z}$ as standard desubstitution schemes, %for some $n\in{\N}$ 
{for some $i'$}.
Applying the definition of right-quasi-recognizability radius to $x=x'=y$, there exist $j,j'\in\cc0R$ such that $\kk_{\co{j-i}\infty}-k_{-i}=\kk'_{\co{j'-i'}\infty}-k_{-i}$. %, or equivalently $\kk_{\co{j-i}\infty}=\kk'_{\co{j'-i'}\infty}$.
{Since $j-i\le R-i\le0$ and $\kk$ is standard, $j-i$ is the number of nonpositive elements in $\kk_{\co{j-i}\infty}$.
  Since $\kk_{\co{j-i}\infty}=\kk'_{\co{j'-i'}\infty}$ and $\kk'$ is standard, this number is also $j'-i'$.
  Thus $j-i=j'-i'$.}
We have proven that for arbitrarily large $i$, $\kk_{\co{j-i}\infty}=\kk'_{\co{j-i}\infty}$.
The fact that for any choice of $j$ as above, $k_{j-i}$ becomes closer and closer to $-\infty$ when $i$ goes to $\infty$ allows us to conclude that $\kk=\kk'$.
We just proved that every $x$ admits at most one desubstitution scheme.

\item[$\boldsymbol{(\Rightarrow)}$] Conversely, assume that $\tau$ is quasi-recognizable.
By a classical compactness argument, {there exists a \textbf{quasi-recognizability radius} $r\in{ \N}$ such that for every standard desubstitution schemes $\kk,\kk'$ of two configurations $x,x'\in X$ such that $x_{\cc{-r}r}=x'_{\cc{-r}r}$, one has $k_0(x)=k_0(x')$.
Let us set $R\defeq\ipart{\frac{r-1+\norm\tau}{\len\tau}}$ and fix two configurations $x,x'$ such that $x_{{ \N}} = x'_{{ \N}}$, with desubstitution schemes $\kk,\kk'$.
  Let $i\in\Z$ be such that $k_i\ge r$.
  One can apply the definition of quasi-recognizability radius to $\sigma^{k_i}(x)$ and $\sigma^{k_i}(x')$ with desubstitution schemes $\sigma^i(\kk)-k_i$ and $\sigma^m(\kk')-k_i$, for $m\defeq\max_{k'_j\le k_i}j$ (see Remark~\ref{r:desubshift}): we get $0=k_i-k_i=k_m-k_i$,
%Since $x_{\N}=x'_{\N}$, $\sigma^{k_i}(x)_{\cc{-r}r}=\sigma^{k_i}(x')_{\cc{-r}r}$, which implies that $0=k_0(\sigma^{k_i}(x))=k_0(\sigma^{k_i}(x')) \in\sett{k'_j}{j\in\Z}-k_i$,
{which means that $k_i \in\sett{k'_m}{m\in\Z}$.}
%This means that for every $i \in \co R\infty$, $k_i$ appears in $\kk'$.
Hence $\sett{k_i}{i\in\Z}\cap\co r\infty\subset\sett{k'_i}{i\in\Z}$.
By symmetry of the condition, and the fact that $k_i$ and $k'_i$ are strictly increasing, we get that $\kk_{\co{j}\infty}=\kk'_{\co{j'}\infty}$, where $j\defeq\min_{k_i\ge r}i$ and $j'\defeq\min_{k'_{i}\ge r}i$.
Since $k_j,k'_{j'}\ge r\ge0\ge k_0$, we have $j,j'\ge0$.
Moreover, Remark~\ref{r:rad} gives that $k_R,k'_R\ge r$, so $j,j'\le R$.
}
\popQED\end{itemize}
%Conversely, assume that $\tau$ has no right-quasirecognizability radius.
%Hence for every $R\in\N$, there exist $x^R,{x'}^R$ with desubstitution schemes $\kk,\kk'$ and $x_\N=x'_\N$ and $n\in\co R{R+\tau}\cap\sett{k_i}{i\in\Z}\setminus\sett{k'_j}{j\in\Z}$ (up to swapping $x$ and $x'$).
%The sequence of pairs $(\sigma^n(x^R,{x'}^R)$, when $R$ goes to infinity, admits a limit $(x,x')$.
\end{proof}

Our definition of quasi-recognizability radius may appear less natural than a `symmetric' version of the quasi-recognizability radius. However it is justified by the following lemma. 

{Let us remind that for any $\Z$-shift $X$ on some alphabet $\A$, and $z' \in \A^{{ \N}}$, we denote by $\pred X{z'}$ the set of left-infinite words $z\in\A^{{ \Z_-}}$ such that $zz'\in X$. For every $\ell \ge 1$, we also denote by $\pred[\ell]X{z'}$ the set of length-$\ell$ words $w$ on $\A$ which are suffix of a word $z \in\A^{{ \Z_-}}$ such that $zz' \in X$. 
}

%{\red Check carefully this lemma and estimates in Remarks leading to it}
\begin{lemma}~\label{l:hgenrad}
  Let us fix some $\Z$-shift $Y\subset\B^\Z$, and a substitution $\tau : \B \to \A^+$ which admits %and $x\in\B^\Z$. %, $y\in\A^\Z$ a desubstitution, and $Y \subset \B^{\Z}$. Let us consider also
    a right-quasirecognizability radius $R$ over $\{x\}\times X$, where %$X\defeq\bigcup_{s\in\Z}\sigma^s\tau(\A^\Z)$.
    $X$ is the set of configurations which admit a desubstitution in $Y$, and $x\in X$. %shift $X=\bigcup_{s\in\Z}\sigma^s\tau(Y)$, 
    Then for every $\ell\ge 1$,
    \[\card{\pred[\ell\len\tau]X{x_{ \N}}}\le\card{\pred[h]X{x_{ \N}}}%\le\max_{y\in Y}\card{\pred[\ell+R]Y{ y_\N}}
      \le\sum_{j'=\spart{R\frac{\len\tau}{\norm\tau}}}^{R}\card\B^{j'+\ell},\]
    where $h\defeq\min_{u\in\lang[\ell]Y}\len{\tau(u)}$.
\end{lemma}
%\comm{Pierre}{bug.
%To solve it, either add the information of $j$ ($\times R$), or $k_0$ ($\times\norm\tau$, but probably optimizable, if one looks at the problematic cases), or allow shorter words (between $\spart{\frac{k_R}{\norm\tau}}\ge\spart{(R-1)\frac{\len\tau}{\norm\tau}+1}$ and $R$): this gives bound $\ell+R$ if $R\in\{0,1\}$, bound $\ell+R+1$ (or less) in general}
\begin{proof}
    Let $\kk$ be a standard desubstitution scheme for $x$.
    Should we swap $x$ for some asymptotic configuration $x'$, we can assume without loss of generality that $k_R$ is the minimal possible $k_R$ over all possible standard desubstitution schemes of all possible $x'$ such that $x'_{{ \N}}=x_{{ \N}}$ (see Remark~\ref{r:rmin}).
    Let $\ell\ge 1$.
     {Let us consider any map $\phi: \pred[h]X{x_{{ \N}}}\to\bigcup_{j'=\spart{R\frac{\len\tau}{\norm\tau}}}^{R}\B^{j'+\ell}$
       such that for any $w \in \pred[h]X{x_{{ \N}}}$, $\phi (w)=y'_{\co{-\ell}{j'}}$, 
%       {\red how do we know that it is defined for $-l$? $h$ is minimal, so there are some substitutions $\tau(u)$ of larger length length than $h$, so do not fit in $x'$. I think we need more explanation in this definition here.
%       Pierre: you're right that $\tau(y'_{\co{-l}{j'}})$ may not fit in $x'_{\co{-h}0}$, but we took any possible extension $x'$ and any desubstitution, so i don't see a problem in the definition itself. Relevance will come from the next paragraph.}
where $y'$ is some standard desubstitution of a configuration $x'$ such that $x'_{\co{-h}\infty}=wx_{{ \N}}$, $\kk'$ the corresponding desubstitution scheme, and $j'\in\cc{\spart{R\frac{\len\tau}{\norm\tau}}}R$ is obtained from Remark~\ref{r:rmin}}.
Beware that this map is not constructively defined; we now simply prove the injectivity of such a map, which will yield our cardinality inequality.
   For every configuration $x'$ such that $x'_{\co{-h}\infty}=wx_{{ \N}}$ and every standard desubstitution scheme $\kk'$ and corresponding desubstitution $y'$, one has
   \[\tau\phi(w)=\tau(y'_{\co{-\ell}{j'}})=x'_{\co{k'_{-\ell}}{k'_{j'}}}=x'_{\co{k'_{-\ell}}{k_R}},\]
   and $k'_{-\ell}\le k'_0-\len{\tau\phi(w)}\le-h$, by definition of $h$ and because $k'_0 \le 0$ since the desubsitution scheme is standard.
   Hence, if $R>0$, then $w=x'_{\co{-h}0}=\tau\phi(w)_{\len{\tau\phi(w)}-k_R+\co{-h}0}$.
   If $R=0$ and $h\ge-k_0$, then $w=x'_{\co{-h}0}=\tau\phi(w)_{\len{\tau\phi(w)}+\co{-h-k_0}0}x_{\co{k_0}0}$.
   If $R=0$ and $h\le-k_0$, then $w=x_{\co{-h}0}$ (there is actually just one such predecessor).
   In all three cases, we have proven that $\phi$ is injective, which straightforwardly gives the second inequality.
     %The second inequality follows from the fact that each $y'_{\co{-\ell}R}$ is in $\pred[\ell+R]Y{y'_{\co R\infty}}=\pred[\ell+R]Y{y_{\co R\infty}}$.
     The first one comes from noticing that $h\ge \ell\len\tau$.
%     The third one from $\pred[$
\end{proof}

\subsubsection{Right-quasi-recognizability radius and finite positive expansiveness}

Using the notion of right-quasi-recognizability radius, we can provide upper bounds on the number of predecessors of 
an element of $\A^{{ \N}}$, and then deduce finitely positively expansiveness for the S-adic limit set under certain conditions. We formulate them in this section.

% \begin{lemma}\label{l:unbounded}
%   Let $\ttau$ be a directive sequence with finite rank $r\in\N$, $x\in\A_0^\Z$ such that $R$ is a common right-quasi-recognizability radius for $\tau_{\co0t}$ over $\{x\}\times\Omega_\ttau$, for every $t\in\N$.
%   Assume that $\ell\in\N$ is such that for every $m\in\N$, there exists $t\in\N$ and a standard desubstitution scheme $\kk$ for $x$ through $\tau_{\co0t}$ such that %for every $x'\in\Omega_\ttau$ and every standard desubstitution scheme $\kk$ for $x'$ through $\tau_{\co0t}$, if $x'_\N=x_\N$, then
%   $k_{R}-\ell\len{\tau_{\co0t}}\le-m$.
%   Then $\card{\pred X{x_\N}}\le r^{\ell}$.
% \end{lemma}
% \begin{proof}
%   Let us show for $m\in\N$, that $\card{\pred[m]X{x_\N}}\le r^{\ell}$, which will give the result, in the limit.
%   Let $t'\in\N$ be such that $k_{R}-\ell\len{\tau_{\co0{t'}}}\le-m$, and $t\ge t'$ such that $\card{\A_t}\le r$.
%   Since $\tau_{\co{t'}t}$ is nonerasing, we still have that $k_{R}-\ell\len{\tau_{\co0{t}}}\le-m$.
%   From Lemma~\ref{l:hgenrad},
%   \[\pred[m]X{x_\N}\le\pred[\ell\len{\tau_{\co0t}}-k_R]X{x_\N}\le\card{\A_t}^{\ell}\le r^\ell.\]
% \end{proof}

\begin{lemma}\label{l:rkrad}
    Let $\ttau$ be an everywhere-growing directive sequence such that there exist $\rk\ge 1,R\in { \N}$ and ${x \in \Omega_{\ttau}}$ such that, for infinitely many $t\in{ \N}$, we have that ${\card{\A_{t+1}}}\le\rk$ and $R$ is a right-quasi-recognizability radius for $\tau_{\co0t}$ over $\{x\}\times\Omega_{\ttau}$.
    Then $\card{\pred{\Omega_\ttau}{x_{ \N}}}\le\sum_{j'=\spart{R\frac{\len\tau}{\norm\tau}}}^R\rk^{j'+1}$.
\end{lemma}
\begin{proof}
    Since $\ttau$ is everywhere-growing, for every $m\ge 1$, there exists $t\in{ \N}$ such that %$(\ell-R)\len{\tau_{\co0t}}
    $\len{\tau_{\co0t}}\ge m$.
    By assumption, there exists $t\ge t'$ such that $\card{\A_{t+1}}\le\rk$ and $R$ is a right-quasi-recognizability radius for $\tau_{\co0t}$ over $\{x\}\times\Omega_{\ttau}$.
%    \pierre{
If $Y=\Omega_{\sigma^t\ttau}$, then the set of configurations which admit a desubstitution in $Y$ is $X=\Omega_\ttau$.
%}{\red Is ${\sigma^t\ttau}$ sequence of substitutions starting with $\tau_t$ instead of $\tau_0$? \pierre{yes; i added the definition after the definition of directive sequence; is it enough?}}
    An application of Lemma~\ref{l:hgenrad} for substitution $\tau_{\co0t}$, configuration $x$ and integer $\ell=1$ then gives
%    {\red
%    \[\card{\pred[m]X{x_\N}}\le\card{\pred[\len{\tau_{\co0t}}]X{x_\N}}\le\card{\A_{t+1}}^{R+1}\le\rk^{R+1}.\]}
    {$$
      \card{\pred[m]{\Omega_\ttau}{x_{ \N}}}\le\card{\pred[1\len{\tau_{\co0t}}]{\Omega_\ttau}{x_{ \N}}}\le\sum_{j'=\spart{R\frac{\len\tau}{\norm\tau}}}^R\card{\A_{t+1}}^{j'+1}\le\sum_{j'=\spart{R\frac{\len\tau}{\norm\tau}}}^R\rk^{j'+1}.
    $$}
    Since this is true for every $m\ge 1$, compactness gives that %$\card{\pred X{x_{\N}}}\le\rk^{R+1}$.
    {$\card{\pred{\Omega_\ttau}{x_{ \N}}}\le\sum_{j'=\spart{R\frac{\len\tau}{\norm\tau}}}^R\rk^{j'+1}$.}
\end{proof}

%{\red Does next theorem have any sense with $\alpha=1$?}

\begin{theorem}\label{t:rkrad}
    Let $\ttau$ be an everywhere-growing directive sequence with finite rank %$\rk\in\Ns$
    {$\rk$}
    such that there exists some $R \in\N$ which is a right-quasi-recognizability radius $R$ for $\tau_{\co0t}$ over $\Omega_{\ttau}$, for every $t\in { \N}$.
    Then the S-adic limit set $\Omega_{\ttau}$ is positively $\rk^{R+1}$-expansive. 
    %if $R\in\{0,1\}$, and positively $\rk^{R+2}$-expansive in general.
    %{\red Do not see why we have 2 cases.}
    %(and so is the classical S-adic system)
\end{theorem}

\begin{proof}
  {This is a direct application of Lemma~\ref{l:rkrad} using that when $\rk\geq 2$, the sum $\sum_{j'=\spart{R\frac{\len\tau}{\norm\tau}}}^R\rk^{j'+1}$ is bounded by $\rk^{R+1}$.}
  {When $\rk=1$, Remark~\ref{r:rk1} gives that the shift is finite, hence positively expansive.}
\end{proof}

\begin{remark}
  The hypothesis in Theorem~\ref{t:rkrad} can be slightly relaxed (for readability we left the theorem written as it is).
For instance, the condition on the right-quasi-recognizability radius can be replaced by the existence of a common right-quasi-recognizability radius $R$ for $\tau_{\co0{t}}$, for infinitely many $t\in{\N}$ such that $\card{\A_{t+1}}\le\rk$.  
\end{remark}

\begin{remark}
Notice also that the quasi-recognizability assumption involves the shift $\Omega_{\ttau}$.
The condition might be hard to check directly, because the set is hard to describe; however, it is implied by
any similar condition where $\Omega_{\ttau}$ is replaced by any shift $X$ which includes it (for instance the full shift), for the reason that if $R$ is a right-quasi-recognizability radius for $\ttau$ on $X$, then it is also on $\Omega_{\ttau}$.
As a matter of fact, many substitutions are actually recognizable on the full shift.
\end{remark}

%The purpose of introducing the higher-generation version of the radius was to obtain this very neat degree for finite positive expansivity.
%One can nevertheless express a natural lower-generation version.
%\begin{theorem}
%    Let $\ttau$ be an everywhere-growing directive sequence with rank $k\in\N$ and such that 
%\end{theorem}

\subsubsection{Suffix codes and uniform substitutions}
\begin{definition}
A \textbf{suffix code} is a set of words none of which is a strict suffix of any other.
\end{definition}

\begin{lemma}\label{l:sufcode}
    Let $\tau:\B\to\A^*$ be a quasi-recognizable substitution such that $\tau(\B)$ is a suffix code.
    Then $1$ is a right-quasi-recognizability radius for $\tau$.
\end{lemma}
\begin{proof}
%Assume that substitution $\tau$ is quasirecognizable and such that $\sett{\tau(a)}{a\in\B}$ is a suffix code.
From Proposition~\ref{p:recog}, $\tau$ has a right-quasi-recognizability radius $R$.
We can assume that it is the minimal one and assume ad absurdum that $R>1$.
Let us consider two configurations $x,x' \in \A^{\Z}$ which admit standard desubstitution schemes $\kk$ and $\kk'$ respectively, and desubstitutions $y,y'$, and satisfy $x_{ \N}=x'_{ \N}$.
By definition there exist some $j,j'\in\cc0R$ such that $k_{\co j\infty}=k'_{\co{j'}\infty}$.
Assume that $x,x'$ witness that $R$ is minimal, meaning that $j=R$ and $k_{R-1}\ne k'_{j'-1}$ or $x_{\co{k_{R-1}}{k_R}}\ne x'_{\co{k'_{j'-1}}{k'_{j'}}}$.
Since $R>1$, we know that $k_{R-1}\ge0$, so that $x_{\co{k_{R-1}}{k_R}}=x'_{\co{k_{R-1}}{k_{j'}}}$.
Having $k_{R-1}=k'_{j'-1}$ would contradict the previous sentence.
On the other hand, having $k_{R-1}<k'_{j'-1}$ would give that the $x_{\co{k_{R-1}}{k_R}}$ admits $x'_{\co{k'_{j'-1}}{k'_{j'}}}$ as a strict suffix.
Since they are both images by $\tau$ of a letter in $\B$, this contradicts the hypothesis.
Symmetrically, having $k_{R-1}>k'_{j'-1}$ would give that the $x_{\co{k_{R-1}}{k_R}}$ is a strict suffix of $x'_{\co{k'_{j'-1}}{k'_{j'}}}$.
This also contradicts the hypothesis.
Hence $R\le1$.
\end{proof}

\begin{corollary}\label{c:sufcode}
    For every directive sequence $\ttau$ which is everywhere-growing, quasi-recognizable, with finite rank %$\rk\in\Ns$
    {$\rk$}, and such that $\tau_t(\A_{t+1})$ is a suffix code for every $t\in{ \N}$, $\Omega_{\ttau}$ is positively $\rk^{2}$-expansive.
\end{corollary}
\begin{proof}
This is direct from Theorem~\ref{t:rkrad} and Lemma~\ref{l:sufcode}.
\end{proof}

{
  The suffix code condition is actually a natural generalization of uniform substitutions, for which we can thus derive a similar statement.
}
\begin{definition}
A substitution $\tau$ is \textbf{uniform} if $\len\tau=||\tau||$, that is, all the images $\tau(a)$ have the same length.
A directive sequence $\ttau$ is \textbf{uniform} if for every $t$, $\tau_t$ is uniform.
In that case, notice that for every $t$, $\tau_{\co0t}$ is also uniform.
\end{definition}

\begin{corollary}\label{c:unif}
    For any directive sequence $\ttau$ which is everywhere-growing, quasi-recognizable on the full shift, uniform, with finite rank $\rk$,
    %{\red $\rk\in\N_2$?,}
     $\Omega_{\ttau}$ is positively $\rk^{2}$-expansive.
\end{corollary}
\begin{proof}
This derives directly from Corollary~\ref{c:sufcode} and the fact that every set of words with uniform length is clearly a suffix code.
\end{proof}

\begin{example}
The substitutive shift corresponding to 
\[\tau:\left|\begin{array}{ccl}a&\mapsto&abc\\b&\mapsto&bbc\\c&\mapsto&aba\end{array}\right.,\]
is positively $9$-expansive. This upper bound does not seem tight, and we {suspect} that a more careful study (of how letters split up between the ones with common suffix and the others) would allow to prove that the hypotheses of Corollary~\ref{c:sufcode} imply $\Omega_\ttau$ are positively $(\ipart{\rk/2}+1)(\spart{\rk/2}+1)$-expansive when $\rk$ is odd.
\end{example}
 %{\red Is it really worth to state it as formal conjecture?}

\begin{example}
    On the contrary, the substitution
\[\tau:\left|\begin{array}{ccl}a&\mapsto&aa\\b&\mapsto&bb\end{array}\right.\]
is a counter-example to Corollary~\ref{c:unif} when dropping the hypothesis of quasi-recognizability over the full shift.
Indeed, $\Omega_\tau$ contains non-periodic asymptotically periodic configurations, for instace $\pinf a\uinf b$, which prevents any finite positive expansiveness, by Corollary~\ref{cor:periodic}.
\end{example}

\subsubsection{Right-marked, return, T\oe plitz substitutions}
%Let us give a first application of Proposition~\ref{p:predsadic}.
In some cases, one can be a little more precise about the finite positive expansiveness degree.
% Let us say that a substitution $\tau:\B\to\A$ is \textbf{suffix-quasi-injective} if there is $q<\len\tau$ such that
% \[\appl{\tau(\B)}{\A^+}{w}%\tau(a)}{\tau(a)_{\len{\tau(a)}
%   {w_{\len w+\co{-q}0%\co{\min(0,\len{\tau(a)}-q)}{\len{\tau(a)}}
%     }}\text{ is one-to-one.}\]
% We may talk about $q$-suffix-quasi-injectivity.
% Note that any suffix-quasi-injective substitution is $(\len\tau-1)$-suffix-quasi-injective.

\begin{definition}
We say that a substitution $\tau:\B\to\A^+$ is \textbf{q-right-recoverable} if it is injective and  $\sett{\tau(a)_{\co q{\len{\tau(a)}}}}{a\in\B}${ is a suffix code}, where $q\in\co1{\len\tau}$. %(or equivalently, the corresponding multiset with multiplicities is a suffix code). 
When there exists $q\in\co1{\len\tau}$ such that $\tau$ is q-right-recoverable, we simply say that it is right-recoverable.
\end{definition}

\begin{remark}\label{rmk:1rrec}
Notice that any $q$-right-recoverable substitution, with $q \ge 1$, is $k$-right-recoverable for every $k \le q$.
\end{remark}

\begin{lemma}\label{lemma:inj.monoid}
    For every $q$-right-recoverable substitution $\tau : \B \rightarrow \A^{*}$, the generated monoid homomorphism $ \tau^* : \B^{*} \rightarrow \A^{*}$ is injective.
\end{lemma}

\begin{proof}
Let us consider two words $u = u_0 \ldots u_m,v = v_0 \ldots v_n \in \A^*$ such that $\tau^*(u) = \tau^*(v)$. As a consequence of this equality, one of the words $\tau(u_m)_{\co q{\len{\tau(u_m)}}}$ and $\tau(v_n)_{\co q{\len{\tau(v_n)}}}$ is suffix of the other. Since the set $\sett{\tau(a)_{\co q{\len{\tau(a)}}}}{a\in\B}${ is a suffix code}, this implies that $\tau(u_m)_{\co q{\len{\tau(u_m)}}} 
= \tau(v_n)_{\co q{\len{\tau(v_n)}}}$ and thus that $\tau(u_m)$ and $\tau(v_n)$ have the same length. Using $\tau^*(u) = \tau^*(v)$ again, we have that $\tau(u_m) = \tau(v_n)$. Since $\tau$ is injective, we have $u_m = v_n$. Finally, we have that 
$\tau^*$ has the same images on $u_0 \ldots u_{m-1}$ and $v_0 \ldots v_{n-1}$. We conclude by a recursion argument.
\end{proof}

\begin{definition}
We call \textbf{extension} of a suffix code $W$ any set $W'$
such that there exists a surjective function $\phi : W \rightarrow W'$ such that for all $w \in W$, $w$ is a suffix of $\phi(w)$. Since $W$ is a suffix code, such a function is also injective. 
\end{definition}

\begin{lemma}\label{lemma:extension.suffcode}
Every extension of a suffix code is a suffix code. 
\end{lemma}

\begin{proof}
Let us consider an extension $W'$ of a suffix code $W$, and 
set $\phi : W \rightarrow W'$ surjective such that for all $w$, $w$ is a suffix of $\phi(w)$. We have to prove that if two words $u',v'$ in $W'$ are such that $u'$ is a suffix of $v'$, then $u'$ and $v'$ are equal. Let us fix two such words $u',v'$. We denote by $u,v$ their respective preimages by $\phi$. As a consequence, 
one of the two words $u,v$ is suffix of the other. Therefore $u=v$. Since $\phi$ is also injective, $u'=v'$.
\end{proof}

\begin{lemma}\label{lemma:inj.suffcode}
    For every suffix code $W$ over alphabet $\B$, and every substitution $\tau : \B \rightarrow \A^{*}$ such that $\tau^*$ is injective, the set $\tau^*(W)$ is a suffix code.
\end{lemma}

\begin{proof}
    It is sufficient to see that for every words $u,v$ such that $u$ is suffix of $v$, $\left(\tau^*\right)^{-1}(u)$ is a suffix of $\left(\tau^*\right)^{-1}(v)$.
\end{proof}

%\begin{definition}
%    { Provided a substitution $\tau:\B\to\A^+$ and %$w$ a word on $\A$, we call $\tau$-\textbf{suffix} of $w$ %any suffix of the form $\tau(u)$.}
%\end{definition}

\begin{lemma}\label{l:sufcompo}
  If $\tau:\B\to\A^+$ is right-recoverable and $\tau':\mathcal{C}\to\B^+$ is $q$-right-recoverable, then $\tau\tau'$ is $q\len\tau$-right-recoverable.
  % If $\tau$ is $q$-suffix-injective and $\tau'$ is $q'$-suffix-injective, then $\tau\tau'$ is $(\max_{u\in\A^{q'-1}}\len{\tau(u)}+q)$-suffix-injective (provided that this is $<\len{\tau\tau'}$.
  % If $\tau'$ is $q'$-suffix-injective, then $\tau\tau'$ is $\max_{u\in\A^{\norm{\tau'}-1}}\len{\tau(u)+q')$-suffix-injective (provided...)
  % If $\tau$ is $q$-suffix-injective, then $\tau\tau'$ is $q'\norm{\tau'}$-suffix-injective (provided...)
\end{lemma}

\begin{proof}
    In order to prove that $\tau \tau'$ is $q\len\tau$-right-recoverable, 
we need to prove that $q\len\tau \in\co1{\len{\tau\tau'}}$, that $W':= \sett{\tau\tau'(c)_{\co{q\len\tau}{\len{\tau\tau'(c)}}}}{c\in\mathcal{C}}${ is a suffix code}, and that $\tau \tau'$ is injective. 
The first point follows from the fact that $\tau'$ is $q$-right-recoverable and thus $q\in\co1{\len{\tau'}}$, and that $\len \tau\len{\tau'} \le \len{\tau\tau'}$.
The third point follows from the fact that $\tau^*$ is injective (Lemma~\ref{lemma:inj.monoid}) and that $\tau'$ is injective.
We are left to prove the second point.
Since $\tau'$ is $q$-right-recoverable, $W\defeq{\sett{\tau'(c)_{\co q{\len{\tau'(c)}}}}{c\in\mathcal{C}}}$ is a suffix code.
  By Lemma~\ref{lemma:inj.suffcode}, $\tau^*(W)$ is a suffix code, and by Lemma~\ref{lemma:extension.suffcode}, $W'$ is a suffix code, as an extension of $\tau^*(W)$.
\end{proof}

\begin{lemma}\label{l:k0bounded}
  Let us fix an integer $q \ge 1$ and $\tau:\B\to\A^+$ be a quasi-recognizable $q$-right-recoverable substitution, and $x\in X\defeq\tau(\B^{\Z}).$
  %\bigcup_{s\in\Z}\sigma^s\tau(\B^\Z)$.
  Then
  \[\card{\pred[q]X{x_{ \N}}}\le\card{\B}.\]
\end{lemma}
\begin{proof}
  Right-recoverability implies the hypothesis of Lemma~\ref{l:sufcode}, so that $1$ is a right-quasi-recognizability radius.
  Let $\kk$ be the unique standard desubstitution scheme for $x$, and $y$ the unique desubstitution corresponding to this desubstitution scheme (it is unique because $\tau$ is injective). We distinguish two cases.
  \begin{enumerate}
  \item \textbf{First case.} Assume that there exists $x'\in X$ such that $x'_{{ \N}}=x_{{ \N}}$ having a standard desubstitution scheme $\kk'$ such that $k'_0>-q$.
    {Because $1$ is a right-quasi-recognizability radius, there are $j,j'\in \{0,1\}$ such that $\kk_{\co j\infty} = \kk'_{\co j\infty}$.
      Since $k_0,k'_0\le0<k_1,k'_1$, we can deduce that $k'_1 = k_1$.}
    If we denote by $y'$ the desubstitution corresponding to $\kk'$, then $\tau(y'_0)=x'_{\co{k'_0}{k'_1}} = x'_{\co{k'_0}{k_1}}$ and $\tau(y_0)=x_{\co{k_0}{k_1}}$.
    Hence $\tau(y'_0)_{\co q{\len{\tau(y'_0)}}}=x_{\co{k'_0+q}{k_1}}$ and $\tau(y_0)_{\co q{\len{\tau(y'_0)}}}=x_{\co{k_0+q}{k_1}}$, so that one is suffix of the other.
    By recoverability, $\tau(y'_0)=\tau(y_0)$.
    Since $\tau$ is injective, $y'_0 = y_0$.
    Hence $0$ is actually also a radius over $\{x\}\times X$.
    We conclude by using Lemma~\ref{l:rkrad}.
  
  \item \textbf{Second case.} On the other hand, suppose that for every $x'\in X$ such that $x'_{{ \N}}=x_{{ \N}}$ and standard desubstitution scheme $\kk'$, one has $k'_0\le-q$.
  Let us consider any map $\phi:\pred[q]X{x_{{\N}}}\to\B$ such that for any $w\in\pred[q]X{x_{{\N}}}$, $\phi(w)=y'_0$ for some desubstitution $y'$ of a configuration $x'$ such that $x'_{\co{-q}\infty}=wx_{{ \N}}$.
  For every configuration $x'$ such that $x'_{\co{-q}\infty}=wx_{{ \N}}$ and every standard desubstitution scheme $\kk'$ and desubstitution $y'$, one has $k'_0\le-q$, by assumption.
  Since $\tau(y'_0)=x'_{\co{k'_0}{k_1}}$, one gets $w=\tau(\phi(w))_{\len{\tau(\phi(w))}-k_1+\co{-q}0}$, which implies that $\phi$ is injective.
  This directly gives the inequality.
 \popQED\end{enumerate}
\end{proof}

\begin{corollary}\label{c:recover}
  Let $\ttau$ be a quasi-recognizable everywhere-growing directive sequence with finite rank $\rk$ such that for every $t'\in { \N}$ and every sufficiently large $t>t'$, ${\tau_{\co{t'}t}}$ is %$1$-
  right-recoverable.
  Then $\Omega_\ttau$ is positively $\rk$-expansive.
\end{corollary}
Because of injectivity, such quasi-recognizable directive sequences will actually be eventually recognizable in the sense of \cite{BST}.
\begin{proof}
  By the everywhere-growing property, for every $m \ge 1$, there exists $t'\in { \N}$ such that $\len{\tau_{\co0{t'}}}\ge m$. 
  %{\red Why $\tau_{\cc0t}$ is right recoverable?}
  By assumption, there exists $t\ge t'$ such that $\A_{t+1}\le\rk$, and $\tau_{\co{t'}t}$ is $1$-right-recoverable.
%  Since $\tau_{\co{t'}t}$ is non-erasing, one still has $\len{\tau_{\co0t}}\ge m$.
  By Lemma~\ref{l:sufcompo}, this implies that $\tau_{\co0t}=\tau_{\co0{t'}}\tau_{\co{t'}t}$ is $\len{\tau_{\co0{t'}}}$-right-recoverable, hence $m$-right-recoverable.
  By Lemma~\ref{l:k0bounded}, we get that for every $x\in\Omega_\ttau$, $\card{\pred[m]X{x_{ \N}}}\le\card{\A_{t}}\le\rk$.
  Since this is true for every $m$, we conclude by compactness that $\card{\pred X{x_\N}}\le\rk$.
\end{proof}

\begin{definition}
We will say that a substitution $\tau$ is \textbf{right-marked} 
when $a \mapsto \tau(a)_{\len{\tau(a)}-1}$ is one-to-one, and that a directive sequence $\ttau$ is \textbf{right-marked} when for every $t \in\N$, $\tau_t$ is right-marked.
In that case, notice that for every $t \in\N$, $\tau_{\co0t}$ is also right-marked. 
\end{definition}

\begin{remark}\label{r:marked}
It is clear that if a substitution is right-marked and expanding, then it is $(\len{\tau}-1)$-right-recoverable. 
Moreover, the composition of two right-marked substitutions is clearly still right-marked.
By telescoping, one can then see that any quasi-recognizable everywhere-growing right-marked directive sequence satisfies the hypothesis of Corollary~\ref{c:recover}.
%Whenever the substitutions $\tau_t$ are all non-erasing, a directive sequence is right-marked if and only if every subsequence $(\tau_{t_i})_{i \in\N}$ is (where $(t_i)$ is an increasing sequence of non-negative integers with $t_0 = 0$).
\end{remark}

\begin{remark}\label{r:markrk}
When a directive sequence $\ttau$ is right-marked, $t \mapsto \card{\A_t}$ is non-increasing, and thus eventually reaches the rank of $\ttau$.
\end{remark}

\begin{corollary}\label{c:marked}
    For every everywhere-growing quasi-recognizable right-marked directive sequence $\ttau$ with finite rank $\rk$ %{\red $\rk\in\N_2$?}
    , $\Omega_{\ttau}$ is positively $\rk$-expansive, and not positively $(\rk-1)$-expansive.
\end{corollary}
\begin{proof}
The fact that it is $\rk$-expansive comes directly from Remark~\ref{r:marked} and Corollary~\ref{c:recover}.
%    Let $v$ be a suffix of a word in $\tau_{\co0t}(\lang{X_{\sigma^t\tau}})$, and $w\in\tau_{\co0t}(\lang{X_{\sigma^t}\tau})$ a minimal left-extension of $v$.
%    If $v$ is the empty word, then $w$ must also be empty.
%    Otherwise $v$ has a last letter $v_{\len v-1}$.
%    Since $\tau_{\co0t}$ is right-marked, there is a unique $a$ such that $\tau_{\co0t}(a)$ ends with $v_{\len v-1}$.
%    So $w=\tau_{\co0t}(w'a)$ for some $w'\in\lang{X_{\sigma^t\tau}}$.
%    If $v$ is a suffix of $w$, then this $w$ was unique.
%    Otherwise, $v$ is a suffix of some $\tau_{\co0t}(v'a)$ for some $v'\in\lang{X_{\sigma^t\tau}}$.
%    Hence one can inductively repeat the argument with $\tau_{\co0t}(v')$ and $\tau_{\co0t}(w')$, and eventually find that $w$ is unique, so that $X_\tau$ satisfies the hypothesis of Proposition~\ref{p:predsadic} with $n=1$.

As Remark~\ref{r:markrk} states, there is $t'\in{ \N}$ such that for every $t\ge t'$, $\card{\A_t} = \rk$.
For any $t\ge t'$, let us fix a one-sided sequence $z^t\in\A_t^{ \N}$, and for each $a\in\A_t$, consider some configuration $x^{t,a}$ such that $x^{t,a}_{-1}=a$ and $x^{t,a}_{ \N}=z^t$.
Since $\tau_{\co0t}$ is right-marked, all the configurations with standard desubstitution scheme $\kk$ for $\tau_{\co0t}$ such that $k_0 = 0$ and with desubstitution $x^{a,t}$ differ at position $-1$, but coincide over ${ \N}$.
By compactness, there exist $r$ configurations in $\Omega_\ttau$ which differ at position $-1$ but coincide over ${ \N}$.
This proves that $\Omega_\ttau$ is not positively $(\rk-1)$-expansive.
%It is straightforward that there exists an increasing sequence of non-negative integers $(t_i)_{i \in\N}$ such that $t_0 =0$ and for every $i$, $\card{\A_{t_i}}=\rk$; there 
%exists a subset $\A$ of $\A_0$ such that for every $i$, the function which to $a \in \A_{t_i}$ associates the last letter of $\tau_{\llbracket 0 , t_{i}\llbracket}(a)$ is a bijection between $\A_{t_i}$ and $\A$.
%This means that for every $a \in \A$, there is exactly one letter $b_i(a)$ in $\A_{t_i}$ such that $\tau_{\llbracket 0,t_i\llbracket}(b_i(a))$ ends with $a$. 
%{ By induction, passing to a subsequence if necessary, we may assume that all $a \in A$ the limit $v^a$ of $\tau_{\llbracket 0,t_i\llbracket}(b_i(a))$ in $\A^{\Z_{-}}$ exists, and in particular $v^a_{-1}=a$. Let us fix any $u\in \A_0^{\N}$ obtained as a limit of words generated by a subsequence of substitutions $\tau_{\llbracket 0,t_i\llbracket}$.}
%Then for every $a \in \A$, $v_a u$ is an element of $\Omega_{\ttau}$. Furthermore, they all differ at position $-1$.
%This means that $u$ has $\rk$ possible extensions in $\Omega_{\ttau}$. As a consequence, this shift cannot be positively $(\rk-1)$-expansive.
\end{proof}
%Positive $\rk$-expansivity could also be proven by noting that right-marked substitutions are exactly those for which $0$ is a right-quasi-recognizability radius.{\red The above seems to be in contradiction with next example. Right?}

\begin{example}
There are right-marked substitutions which are not quasi-recognizable, for instance:
\[\tau:\left|\begin{array}{ccl}a&\mapsto&aa\\b&\mapsto&bb\end{array}\right. \text{ and } \ \tau':\left|\begin{array}{ccl}a&\mapsto&aa\\b&\mapsto&ab\end{array}\right..\]
The corresponding substitution limit sets are not finitely positively expansive, because they admit some non-periodic asymptotically periodic configurations, like $\pinf a\uinf b$ for the first one and $\pinf ab\uinf a$ for the second.
As a consequence, it is not possible to drop the quasi-recognizability hypothesis in Corollary~\ref{c:marked} (and Corollary~\ref{c:recover}).
\end{example}

% \begin{remark}
% The proof of Corollary~\ref{c:marked} can be generalized to everywhere-growing, quasi-recognizable on the full shift, and finite-rank directive sequences $\ttau$ on ${\A}$ such that the set $\sett{\tau_t(a)_{\co1{\len{\tau_t(a)}}}}{a\in\A_{t+1}}$ is a suffix code. 
% %(that is, they form a set of words, none of which is a strict suffix of any other).
% In particular right-marked directive sequences satisfy this property, but there are  many others. %, such as left-proper {\red what is left proper?} uniform ones).
% %In that case, note that there is $f(t)\to\infty$ such that $(\tau_{\co0t}(a)_{\co{f(t)}[\len{\tau_t(a)}})_{a\in\A_{t}$ is a suffix code.
% The general proof requires pointwise versions of the definition of radius and of Lemma~\ref{l:hgenrad}.
% %\begin{lemma}~\label{l:hgenrad}
% %    Let $\tau$ be a substitution, $R$ its right-quasirecognizability radius for $x\in X=\bigcup_{s\in\Z}\sigma^s\tau(Y)$, $\ell\in\N$, and $m_l=\min_{u\lang[\ell]Y}\len{\tau(u)}$.
% %    Then \[\pred[\ell\len\tau]X{x_\N}\le\pred[m_l]X{x_\N}\le\max_{z\in Y}\pred[\ell+R]Y{z_\N}.\]
% %\end{lemma}
% \end{remark}

Here are some examples of applications of Corollary~\ref{c:marked}.
\begin{example}~\label{x:fibotm}
The Fibonacci ($\tau(a)=ab$ and $\tau(b)=a$) and Thue-Morse ($\tau(a)=ab$ and $\tau(b)=ba$) substitutions  are clearly right-marked and everywhere-growing and they are known to be quasi-recognizable \cite{mosse}.
Hence the two corresponding S-adic limit sets, and thus the two corresponding substitutive shifts 
%(which are also equal to the classical substitutive shifts) 
are both positively $2$-expansive.
\end{example}

Another very classical class of substitutions connected to Corollary~\ref{c:recover} are return substitutions.
\begin{definition}
$\tau:\B\to\A^*$ is a \textbf{return substitution} with respect to some word $w\in\A^+$ if it is injective and $w$ appears exactly twice in each $\tau(a)w$, for $a\in\B$, once as a prefix (and once as a suffix).
A substitution $\tau:\B\to\A^*$ is \textbf{left-proper} if for every $a,b\in\B$, $\tau(a)_0=\tau(b)_0$.
A word $w\in\A^*$ is \textbf{nonoverlapping} if $\A^\ell w\cap w\A^\ell=\emptyset$ for every $\ell\in\co1{\len w}$.
\end{definition}
\begin{remark}~\label{r:proper}
    Every return substitution is %injective and
    left-proper.
    Every injective left-proper substitution is $1$-right-recoverable.
    Every return substitution with respect to some word of length at least $2$ is expanding.
\end{remark}

\begin{proposition}\label{p:retrec}
    Let $\tau:\B\to\A^*$ be a return substitution with respect to some nonoverlapping word $w$.
 %   \begin{enumerate}
 %   \item If $\len w\ge2$, then $\tau$ is expanding.
 %   \item If $w$ is nonoverlapping, t
 Then $\tau$ is recognizable.
 %   \end{enumerate}
\end{proposition}
%The classical definition of return substitution additionally involves injectivity, and this proposition can be adapted and prove, with this stronger assumption, recognizability instead of quasi-recognizability (but we do not use this in this article).
\begin{proof}
%    The first point is clear, because every $\tau(a)$ starts with $w$.
    Since $\tau$ is injective, it is enough to prove quasi-recognizability.
    Suppose that there exists $x\in\A^\Z$ with two standard desubstitution schemes $\kk$ and $\kk'$.
    %One can assume, should one translate the configuration, that $k_0\ne k'_0$ and, should one swap the two schemes, that $k_0<k'_0$.
    Since the images of letters all start with $w$, we get for every $i,j\in\Z$, $x_{\co{k_i}{k_i+\len w}}=w=x_{\co{k'_j}{k'_j+\len w}}$.
    Since $w$ is nonoverlapping, one can deduce that $k'_j\notin\co{k_i}{k_i+\len w}$.
    Suppose that there exists $i\in\Z$ such that $k_i\notin\sett{k'_j}{j\in\Z}$, and let $j$ be maximal such that $k'_j<k_i$.
    This maximality and the previous sentence give that $k'_{j+1}\ge k_i+\len w$.
    There exists $a\in\B$ such that $\tau(a)=x_{\co{k'_j}{k'_{j+1}}}$, but $\tau(a)w$ contains $x_{\co{k_i}{k_{i}+\len w}}=w$ as a strict factor, which contradicts the definition of return substitution.
    {We have proved that $\sett{k_i}{i\in\Z} \subset \sett{k'_j}{j\in\Z}$. By symmetry we have equality. Since the two desubstitution schemes are standard, they are equal.}
\end{proof}

%\begin{theorem}
%    A shift is minimal aperiodic if and only if it is the S-adic system built from some expanding primitive directive sequence of return substitutions with respect to nonoverlapping words.
%\end{theorem}
%\begin{proof}
%
%\end{proof}

\begin{corollary}\label{c:return}
    Let $\ttau$ be an everywhere-growing directive sequence of return substitutions with respect to nonoverlapping words, with finite rank $\rk$.
    Then $\Omega_\ttau$ is positively $\rk$-expansive.
\end{corollary}
\begin{proof}
    Proposition~\ref{p:retrec} gives that $\ttau$ is quasi-recognizable.
    Remark~\ref{r:proper} implies that all {return} substitutions are $1$-recoverable.
    Corollary~\ref{c:recover} can then be applied.
\end{proof}
    If one drops the injectivity assumption from the definition of return substitution, one can prove quasi-recognizability (instead of recognizability), and still obtain positive $\rk^2$-expansiveness, thanks to the suffix code property and Corollary~\ref{c:sufcode}.
    %{\red (Piotr) This remark is very sketchy, For this we need to adjust also Proposition~\ref{p:retrec},
    %since in Corollary~\ref{c:sufcode} we need quasi-recognizability.}
  %\pierre{i detailed a word on this in the sentence. It's sketchy but true (i had written the proof in a preliminary version, but i'm not sure we want to make the article longer).}

Example~\ref{x:fibotm} can be generalized to the class of Arnoux-Rauzy shifts (sometimes called episturmian, with some variants), which include all sturmian shifts.
One way to define them (see \cite{andrieu}) is as S-adic $\Z$-shifts corresponding to aperiodic directive sequences of (finitely many) substitutions of the form:
\[\tau_i:\left|\begin{array}{ccl}a_j&\mapsto&a_ia_j\\a_i&\mapsto&a_i\end{array}\right.,\]
where $i\in\co0\rk$ and $\rk\ge 2$ is the alphabet size.
\begin{corollary}\label{c:arauzy}
    Every Arnoux-Rauzy shift over $\rk$ letters is positively $\rk$-expansive (and not positively $(\rk-1)$-expansive).
\end{corollary}
\begin{proof}
    This directly comes from Corollary~\ref{c:marked}, the fact that substitution $\tau_i$ is a right-marked return substitutions with respect to nonoverlapping word $a_i$, which implies recognizability by Proposition~\ref{p:retrec}, and that the directive sequences involving infinitely many of at least two distinct substitutions are everywhere-growing.
\end{proof}

{
Corollary~\ref{c:return} can be applied to much more general class of subshifts than Arnoux-Rauzy shift. By classical results \cite{sadicmin}, every minimal $\Z$-shift can be expressed as an S-adic $\Z$-shift corresponding to an everywhere-growing directive sequence of return substitutions with respect to words of length $1$ (which are clearly nonoverlapping).  Unfortunately, this classical construction often yields an infinite-rank directive sequence (even in cases when the shift could also correspond to a finite-rank one), which prevents application of our results.
Indeed, (infinite-rank) minimal shifts may for instance have positive entropy %(see for instance \cite[Theorem~4.77]{kurka}),
and thus cannot be finitely positively expansive, by Theorem~\ref{thm:entropy.zero.gen}.
%finite-rank shifts have entropy $0$ \cite[Corollary~6.5]{donoso}
Yet, our corollary applies to many minimal $\Z$-shifts, including examples presented earlier in this text.
In particular, there are some criteria for minimal $\Z$-shifts to have finite rank, such as having non-superlinear complexity \cite[Theorem~5.5]{donoso}.
Interested readers are referred to \cite{donoso} and references therein for more details. In what follows we will provide an example of technique of this kind for T\oe plitz shifts.
%morphismes de retour -> dendric ternary minimal, linearly recurrent, minimal
}

\begin{definition}
  A substitution is said to be a \textbf{T\oe plitz} substitution if it is left-proper, uniform, injective and expanding.
    A \textbf{T\oe plitz} limit set is an the limit set of a directive sequence of T\oe plitz substitutions.
\end{definition}
{
  Note that the expanding assumption is automatic whenever the alphabet is non-trivial.
An important remark is that the composition of finitely many T\oe plitz substitutions is also a T\oe plitz substitution.
}

T\oe plitz shifts are more commonly defined through their quasi-periodic structure, but the definitions are equivalent (see \cite[Prop 4.70]{kurka} for instance).
{
  Let us prove a preliminary result.
  A directive sequence is said to be \textbf{weakly primitive} if:
    \[\forall t\in{ \N},\exists t'\ge t,\forall a\in\A_{t'},{\forall b\in\A_{t}},b\text{ appears in }\tau_{\co t{t'}}(a).\]
\begin{proposition}\label{p:topmin}~\begin{enumerate}
  \item Every T\oe plitz limit set $\Omega_\ttau$ is aperiodic, unless for all sufficiently large $t\in{ \N}$, there exists $a_t\in\A_t$ such that $\tau_t(a_{t+1})=a_t^{\len{\tau_t}}$.
  \item Every weakly primitive T\oe plitz limit set is minimal (in particular it is equal to the classical T\oe plitz shift). 
\end{enumerate}\end{proposition}
\begin{proof}~\begin{enumerate}
  \item Let $p\ge 1$ be the minimal possible period for all possible configurations $x\in\Omega_\ttau$.
    Up to a shift, one can assume that $x=\tau_0(y)$ for some $y\in\Omega_{\sigma\ttau}$.
    \begin{itemize}\item If $p$ is prime with $\len{\tau_0}$, then for every $i\in\Z$, there exists $k,\ell\in\Z$ such that $i+kp=\ell\len{\tau_0}$.
    We obtain that $x_i=x_{i+kp}=x_{\ell\len{\tau_0}}$ is the first symbol $a_0$ of the common prefix, i.e. $x=\dinf{a_0}$, and $p=1$ by minimality.
    By injectivity, one gets that $y=\dinf{a_1}$ for some symbol $a_1\in\A_1$, and iteratively a sequence of symbols $a_t\in\A_t$ such that $\tau_t(a_{t+1})=a_t^{\len{\tau_t}}$.
    \item Otherwise, the least common multiple $p\wedge\len{\tau_0}$ is also a period for $x$, and by injectivity of $\tau_0$, $p\wedge\len{\tau_0}/\len{\tau_0}$ is a period for $y$.
      Hence $\Omega_{\sigma\ttau}$ admits a smaller period than $p$, and the other item will be eventually reached.
    \end{itemize}
  \item Let us prove that every $u\in\lang{\Omega_\ttau}$ appears in every configuration $x\in\Omega_\ttau$.
    Since $\ttau$ is expanding, there exists $t\in{ \N}$ such that $\len u\le\len{\tau_{\co0t}}$.
    By definition, $u$ appears in some $\tau_{\co0t}(v)$, and by the previous inequality, $v$ can be assumed to have length $1$ or $2$.
    If $\len v=1$, primitivity gives some $t'\ge t$ such that $v$ appears in $\tau_{\co t{t'}}(a)$ for every letter $a\in\A_{t'}$.
    In particular, $v$ appears in every word of length $2\len{\tau_{\co t{t'}}}-1
    $ in $\Omega_{\sigma^t\ttau}$, hence $u$ appears in every word of length $2\len{\tau_{\co0{t'}}}
    $ in $\Omega_\ttau$.
    If, on the other hand, $\len v=2$, then either $v$ appears in some $\tau_t(a)$ for some $a\in\A_t$ or $v=\tau_t(a)_{\len{\tau_t}-1}\tau_t(b)_0$.
    In both cases, we can apply the previous argument to $a$ instead of $w$ and $t+1$ instead of $t$, so that, for some $t'\ge t+1$, $a$ appears in every word of length $2\len{\tau_{\co{t+1}{t'}}}-1
    $ in $\Omega_{\sigma^{t+1}\ttau}$.
    In the first case, we get that $u$ appears in every word of length $2\len{\tau_{\co0{t'}}}
    $ in $\Omega_\ttau$.
    In the second case, note that $\tau_t(b)_0$ is the common first letter for every $\tau_t(c)$, where $c\in\A_{t+1}$.
    Hence $\tau_t(a)\tau_t(b)_0$ appears in every word of length $2\len{\tau_{\co t{t'}}}
    +1$ in $\Omega_{\sigma^t\ttau}$, and we obtain that $u$ appears in every word of length $2\len{\tau_{\co0{t'}}}+\len{\tau_{\co0t}}$ in $\Omega_{\sigma^t\ttau}$.
  \popQED\end{enumerate}\end{proof}
}

\begin{corollary}\label{c:top}
    Every recognizable {limit set of T\oe plitz directive sequences} with finite rank $\rk\ge 1$ is positively $\rk$-expansive.
\end{corollary}
\begin{proof}
    The statement directly derives from Corollary~\ref{c:recover} and (the second statement of) Remark~\ref{r:proper}.
    %The T\oe plitz shift is included in it.
\end{proof}

There are criteria for recognizability of T\oe plitz $\Z$-shifts, like the following one.
%Note that there is always a non-empty maximal common prefix $u$ between letter images.
For every finite word $u$, we call \textbf{period} of $u$ any $p\ge 1$ such that for every $i<\len u$ with $i+p<\len u$, we have $u_i=u_{i+p}$. %,j < |u|$ such that $j>i$, $u_j=u_{i}$ whenever $p | (j-i)$.

For a left-proper substitution, we call \textit{maximal common prefix} of this substitution the maximal non-empty word which is prefix of all of its images.

\begin{proposition}\label{p:toprec}
  Assume that $\tau:\B\to\A^*$ is a T\oe plitz substitution whose maximal common prefix $u$ has smallest period $p$.
  Then $\tau$ is recognizable over the set $\tau(X)$, where $X$ is the set of $x \in \A^\Z$ satisfying the three following properties: there exists $i\in\Z$ such that $\tau(x_i)$ does not start by $uu_{\co{\len u-p}{\len u}}$; there exists $i\in\Z$ such that $\tau(x_i)$ does not end by $u_{\co0p}$; and there exists $i\in\Z$ such that $\tau(x_i)$ does not contain an occurrence of $u$ disjoint from its prefix $u$.
  %the set of $\tau(x)$ such that $\tau(x_i)$ does not start by $uu_{\co{\len u-p}{\len u}}$, for some $i\in\Z$, $\tau(x_i)$ does not end by $u_{\co0p}$, for some $i\in\Z$, and $\tau(x_i)$ does not contain a disjoint copy of $u$, for some $i\in\Z$.

  In particular, if $u$ is nonoverlapping and has length $\len u\ge\len\tau/2$, then $\tau$ is recognizable over the set of aperiodic configurations.
\end{proposition}
%counterexample plus grand ? : 0⟼uau et 1⟼uub
Before proving the proposition, let us state some consequence.
\begin{corollary}
  Let $\ttau$ be a directive sequence of T\oe plitz substitutions with rank $\rk\ge 1$, such that for every $t\in{ \N}$, the maximal common prefix of $\tau_t$ is nonoverlapping and has length at least $\len{\tau}/2$.
  Then $\Omega_\ttau$ is positively $\rk$-expansive.
\end{corollary}
%+ for infinitely many $t$?
\begin{proof}
  The fact that the maximal common prefix is nonoverlapping makes it impossible to have that for all $t\in{ \N}$ sufficiently large, there exists $a_t\in\A_t$ such that $\tau_t(a_{t+1})=a_t^{\len{\tau_t}}$. Indeed, under this hypothesis for all $t$ the maximal 
common prefix of $\tau_t$ would be of the form $a_t^{p_t}$, where $p_t$ is an integer, which implies that it is overlapping. By the first point of Proposition~\ref{p:topmin}, $\Omega_{\boldsymbol{\tau}}$ is aperiodic. 
Hence Proposition~\ref{p:toprec} gives that $\ttau$ is recognizable, and we can conclude by Corollary~\ref{c:top}.
\end{proof}
\begin{proof}[Proof of Proposition~\ref{p:toprec}]
  In order to prove that $\tau$ is recognizable, since $\tau$ is assumed to be injective, it is enough to prove quasi-recognizability.

    Let $y$ be a configuration of $\tau(X)$ which has $x$ as desubstitution such that $x$ satisfies the properties in the statement. By shifting, we can assume that $y=\tau(x)$, so that the corresponding standard desubstitution scheme is $(i\len\tau)_{i\in\Z}$.
    Assume that it admits a distinct standard desubstitution scheme $\kk=(k_0+i\len\tau)_{i\in\Z}$, where $k_0<0$. We conclude according to the value of $k_0$:
    \begin{enumerate}
		\item \textbf{If} $\boldsymbol{-\len u\le k_0<0}$, then $y_{\co{k_0}{\len u}}\in u\A^{-k_0}\cap\A^{-k_0}u$, which means that $-k_0$ is a period for $u$, hence a multiple of $p$.
    Moreover, for every $i\in\Z$, $\tau(x_i)$ ends with
    \[\tau(x_i)_{\co{\len\tau+k_0}{\len\tau}}
    = y_{\co{(i+1)\len\tau+k_0}{(i+1)\len\tau}}
    = y_{\co{k_{i+1}}{k_{i+1}-k_0}}
    =u_{\co0{-k_0}}.\]
Since $-k_0$ is a multiple of $p$, and $p$ is a period of $u$, $u_{\co0{-k_0}}$ ends with $u_{\co0{p}}$. This implies that for all $i$, $\tau(x_i)$ ends with $u_{\co0{p}}$, which contradicts the second assumption in the statement.
    %which contradicts our first assumption, since $-k_0$ is a period.
    
    \item If $\boldsymbol{\len u-\len\tau\le k_0<-\len u}$. For every $i\in\Z$, $\tau(x_i)$ contains the word
    \[\tau(x_i)_{\co{\len\tau+k_0}{\len\tau+k_0+\len u}}
    =y_{\co{(i+1)\len\tau+k_0}{(i+1)\len\tau+k_0+\len u}}
    =y_{\co{k_{i+1}}{k_{i+1}+\len u}}
    =u.
    \]
    Since $\len\tau+k_0\ge\len u$ this other occurrence of $u$ is disjoint from the prefix $u$, which contradicts the third assumption.
    
    \item If $\boldsymbol{-\len\tau<k_0<\len u-\len\tau}$. Then $0<k_1<\len u$, and thus
    $\tau(x)_{\co0{k_1+\len u}}\in u\A^{k_1}\cap\A^{k_1}u$, which means that $k_1$ is a period for $u$, hence a multiple of $p$.
    For every $i\in\Z$, $\tau(x_i)_{\co{\len u}{\len\tau}}$ starts with
    \[\tau(x_i)_{\co{\len u}{\len u+k_1}}
    = y_{\co{i\len\tau+\len u}{i\len\tau+\len u+k_1}}
    = y_{\co{k_{i+1}-k_1+\len u}{k_{i+1}+\len u}}
    =u_{\co{\len u-k_1}{\len u}},\]
    which contradicts the first assumption, since $k_1$ is a period.

\end{enumerate}
%    The only possibility left is that $k_0=0$, so that $\kk=(i\len\tau)_{i\in\Z}$ is the unique possible standard desubstitution scheme.
%    This property is preserved when shifting $\tau(x)$
%    {, i.e. $\kk=(i\len\tau-j)_{i\in\Z}$ is the unique possible standard desubstitution scheme of $\sigma^j(\tau(x))$ for any $j\in \Z$.}.

    Let us now prove the second part of the statement. Assuming that $u$ is nonoverlapping, then its smallest period is $\len u$, so that the first and second hypotheses are actually implied by the third one.
    It is thus sufficient to prove it and then apply the first part of the statement. When $\len u>\len\tau/2$ it is straightforward. When $\len u=\len\tau/2$ on the other hand, 
for all $y$ aperiodic and $x$ a desubstitution of $y$, then $x$ is aperiodic as well and thus there 
 	exists $i\in\Z$ such that $\tau(x_i)\ne uu$. This $i$ satisfies the third hypothesis.
\end{proof}

The following extends the result of C.~Morales on existence for every $n \ge 2$ of systems which are positively $n$-expansive and not $(n-1)$-positively expansive, providing examples within minimal dynamical systems. %Our constructions will be based on result obtained in Section~\ref{section:sufficient}.
\begin{proposition}
For every $n \ge 2$, there exists a minimal substitutive T\oe plitz shift over alphabet $\co0n$ which is positively $n$-expansive and not positively $(n-1)$-expansive.  
\end{proposition}

\begin{proof}
  Let us consider substitution $\tau$ over $\co0n$ defined by $\tau \colon k \mapsto 0 \ldots (n-1) k$ for each symbol $k$.
  This T\oe plitz substitution satisfies all the hypotheses from Propositions~\ref{p:topmin} and~\ref{p:toprec}, hence it is recognizable, everywhere-growing and $\Omega_\tau$ is minimal.
  Moreover, the substitution is also right-marked, hence Corollary~\ref{c:marked} implies that it is positively $n$-expansive {and not $(n-1)$-expansive}.
\end{proof}
%\begin{remark}
%Another way to prove this is to consider the regular T\oe plitz shift 
%which at level $p$ is marked with the class of $p$ modulo $n+1$.
%\end{remark}

\subsubsection{Right-quasi-recognizability radius growth}
Let us derive some additional applications from Theorem~\ref{t:rkrad}. %in cases which are not as clean and classical as the previous ones.

When two substitutions $\tau$ and $\tau'$ can be composed, it is clear that $\len{\tau \tau'}\le\len\tau\len{\tau'}$.
It is natural to wonder how 
composition acts on the right-quasi-recognizability radius that we have defined above.

\begin{lemma}\label{l:comprad}
Let $\tilde\tau:\B\mapsto\mathcal C$ be a substitution with right-quasi-recognizability radius $\tilde R$ over $X$, and $\tau:\A\to\B$ be an injective substitution with right-quasi-recognizability radius $R$ over $\tilde\tau(X)$.
Then $\spart{\frac R{\len{\tilde\tau}}+\tilde R}$ is a right-quasi-recognizability radius for $\tau\tilde\tau$.
%Let $t\ge1$ and, for $i\in\co0t$, $\tau_i$ be a substitution from $\A_{i+1}$ to $\A_i$, admitting a right-quasirecognizability radius $R_i$ over $\bigcup_{s\in\Z}\sigma^s\tau_{\co it}(X)$.
%Then $\sum_{i<t}\frac{R_i}{\len{\tau_{\oo it}}}$ and $\frac1{\len{\tau_{\co0t}}}\sum_{i<t}R_i\len{\tau_{\cc0i}}$ are right-quasirecognizability radius for the composition $\tau_{\co0t}$ over $\bigcup_{s\in\Z}\sigma^s\tau_{\co0t}(X)$.
\end{lemma}
\begin{proof}
%Let us prove the statement by recurrence.
%It is obvious if $t=1$.
%Assume that it is true for some $t\ge1$, and let us prove it for $t+1$.
%Let $x,x'\in\bigcup_{s\in\Z}\tau_{\co0t}(X)$ with standard decomposition schemes $\kk$ and $\kk'$ for $\tau_{\co0t}$, such that $x_\N=x'_\N$.
%By the recurrence hypothesis, one knows that \TODO{}
Let $x,x'$ be two configurations with $x_{{ \N}}=x'_{{ \N}}$.
Any standard desubstitution scheme of $x$ for $\tau\tilde\tau$ can be written as $(k_{\tilde k_i})_{i\in\Z}$, where $\kk$ is a standard desubstitution scheme of $x$ for $\tau$, $y$ is a corresponding desubstitution, and $\tilde\kk$ is a standard desubstitution scheme of $y$ for $\tilde\tau$.
Similarly, any desubstitution scheme of $x'$ for $\tau\tilde\tau$ can be written as $(k'_{\tilde k'_i})_{i\in\Z}$.
%Consider such $k,\tilde k,k',\tilde k'$ such that $(k_{\tilde k_i})_{i\in\N}$, $(k'_{\tilde k'_i})_{i\in\N}$, as well as $k$ and $k'$ are standard.
Since $R$ is a right-quasi-recognizability radius for $\tau$, there exist $j,j'\le R$ such that $k_{\co j\infty}=k'_{\co{j'}\infty}$.
By injectivity of $\tau$, one has $\sigma^j(y)_{{ \N}}=\sigma^{j'}(y)_{{ \N}}$.
Following Remark~\ref{r:desubshift}, let us define $l\defeq\max_{\tilde k_i\le j}i\le\spart{\frac j{\len\tau}}$, and $l'\defeq\max_{\tilde k'_i\le j'}i\le\spart{\frac{j'}{\len\tau}}$, which make $\sigma^l\tilde k-j$ a standard desubstitution scheme for $\sigma^j(y)$ and $\sigma^{l'}\tilde k'-j'$ a standard desubstitution scheme for $\sigma^{j'}(y')$.
Since $\tilde R$ is a right-quasi-recognizability radius for $\tilde\tau$, there exist $\tilde j,\tilde j'\le\tilde R$ such that $\tilde k_{\co{l+\tilde j}\infty}-j=\tilde k'_{\co{l'+\tilde j'}\infty}-j'$.
For $n\in{ \N}$, we get that
\[k_{\tilde k_{l+\tilde j+n}}=k_{\tilde k'_{l'+\tilde j'+n}-j'+j}=k'_{\tilde k'_{l'+\tilde j'+n}}.\]
Since $l+\tilde j$ and $l'+\tilde j'$ are smaller than $\spart{\frac R{\len{\tilde\tau}}}+\tilde R$.
Moreover, $x_{\co{k_j}\infty}=x'_{\co{k'_{j'}}\infty}$, $l+\tilde j\ge j$, and $l'\tilde j'\ge j'$, so that we also have $x_{\co{k_{\tilde k_{l+\tilde j}}}\infty}=x'_{\co{k'_{\tilde k'_{l'+\tilde j'}}}\infty}$. We thus have proved that $\spart{\frac R{\len{\tilde\tau}}}+\tilde R$ is a right-recognizability radius for $\tau\tilde\tau$.
\end{proof}

\begin{corollary}~\label{c:serie}
Let $\ttau$ be an everywhere-growing directive sequence with finite rank $\rk\ge 1$ such that there exists $m\ge 1$ and $(R_i)_{i\in{ \N}}$ such that for every $i\in{ \N}$, $R_i$ is a right-quasirecognizable radius for $\tau_i$ over the shift $\Omega_{\sigma^i(\ttau)}$ with: 
\[\sum_{i<t}R_i\len{\tau_{\cc0i}}\le \len{\tau_{\co0t}}m.\]
Then $\Omega_\ttau$ is positively $\rk^{m+2}$-expansive.
\end{corollary}
\begin{proof}
By applying Lemma~\ref{l:comprad} inductively, for every $t$, $\spart{\sum_{i<t}\frac{R_i}{\len{\tau_{\oo it}}}}$ is a right-quasi-recognizability radius for $\tau_{\co0t}$. Since for every $i,t$, $\len{\tau_{\co0t}}\le\len{\tau_{\cc0i}}\len{\tau_{\oo it}}$, we have: 
\[\spart{\sum_{i<t}\frac{R_i}{\len{\tau_{\oo it}}}}
    \le\spart{\frac1{\len{\tau_{\co0t}}}\sum_{i<t}R_i\len{\tau_{\cc0i}}}
    \le m.\] 
Thus $m$ is a right-quasi-recognizability radius for $\tau_{\co0t}$ for every $t$. The statement follows from Theorem~\ref{t:rkrad}.
\end{proof}

\begin{corollary}\label{c:bounded}
Let $\ttau$ be a directive sequence and $\rho>1$ such that $\len{\tau_{\co0t}}\sim\rho^t$ and there exists $R$ which is a right-quasi-recognizability radius of $\tau_t$ over $\Omega_{\sigma^t(\ttau)}$, for every $t\in{ \N}$.
Then $\Omega_\ttau$ is finitely positively expansive.
\end{corollary}
% The first assumption %corresponds to having only exponentially-growing letters. It is the case 
% is satisfied by many everywhere-growing directive sequences, which satisfy strong properties such as: every $\tau_i$ is expanding;or every $\tau_{\co i{i+t}}$ is expanding, for some fixed $t\in\Ns$.
\begin{proof}
From the hypotheses, we have that there exists $\beta>0$ such that 
\[\sum_{i<t}{R}\frac{\len{\tau_{\co0i}}}{\len{\tau_{\co0t}}}\le\sum_{i<t}R\beta\rho^{i-t}=\frac{\beta R(1-\rho^{-t})}{\rho-1} < \frac{\beta R}{\rho-1}.\]
We conclude by Corollary~\ref{c:serie}.
\end{proof}

{
The second hypothesis of Corollary~\ref{c:bounded} is probably difficult to check because it involves both the different substitutions and the different partial limit sets $\Omega_{\sigma^t\ttau}$. 
%However it is natural regarding the tools developed in the present paper.
It is quite natural though (and sometimes even assumed in the definition of S-adic) to consider only finitely many substitutions.

A directive sequence is \textbf{finitary} if it involves only finitely many distinct substitutions.
\begin{corollary}\label{c:finitary}
  Let $\ttau$ be a finitary directive sequence of substitutions such that $\len{\tau_{\co0t}}$ grows exponentially.
  If either all of these substitutions are quasi-recognizable, or $\Omega_\ttau$ is aperiodic, then $\Omega_\ttau$ is finitely positively expansive.
\end{corollary}
\begin{proof}
  The first statement comes from Corollary~\ref{c:bounded} and the existence of a common right-quasi-recognizability to any finite set of quasi-recognizable substitutions, by Proposition~\ref{p:recog} and taking the maximum.
  The second statement comes additionnally from the fact that all substitutions are recognizable over aperiodic configurations, by \cite{BST,beal_recognizability_2021}.
\end{proof}
}
\begin{remark}
  The assumptions of the last corollaries are quite natural, because many substitutions are recognizable (see \cite{mosse,BST}), and frequently assumed properties of finitary directive sequences, imply exponential growth.
  For instance, \textbf{strong primitivity}: %, where $\ttau$ is strongly primitive when:
    \[\exists j\in{ \N},\forall t\in{ \N},\forall a\in\A_{t+j}, {\forall b\in\A_{t}},b\text{ appears in }\tau_{\co t{t+j}}(a).\]
\end{remark}

\begin{remark}
    It is to be noted that every everywhere-growing S-adic shift is eventually recognizable, in the sense that $\sigma^t\ttau$ is recognizable for some $t\in{ \N}$ \cite[Theorem~5.2]{BST}.
From this it can be deduced that they are all factors of recognizable S-adic shifts \cite{donoso}.
But the transient part (or the factor map) could in theory yield infinitely many left tails for the same configuration right half.
In the case of substitutive shifts, though, it is rather easy to derive finite positive expansiveness.
\end{remark}

\begin{corollary}\label{c:subst}
    Let $\ttau$ be a preperiodic everywhere-growing directive sequence %(for instance purely substitutive) 
    such that the shift $\Omega_\ttau$ is aperiodic. 
    Then $\Omega_\ttau$ is finitely positively expansive.
\end{corollary}
%{\red Was purely substitutive defined?}
\begin{proof}
    It is rather straightforward (see for instance \cite{pytheas}) that, if a preperiodic directive sequence $\ttau$ is everywhere-growing, then $\len{\tau_{\co0t}}\sim\rho^t$ for some $\rho>1$.
    It is also known from \cite[Theorem~5.3]{BST} that it is recognizable provided that $\Omega_\ttau$ is aperiodic; there is a bounded radius thanks to Lemma~\ref{l:comprad} and the finite number of distinct limit sets $\Omega_{\sigma^t\ttau}$ (thanks to preperiodicity of the sequence $\ttau$).
    The statement follows from Corollary~\ref{c:bounded}.
\end{proof}

\begin{remark}
  % COMPLETELY FALSE
  %Note also that if every $\tau_t$ has a maximal common prefix $u$ of length $\len{\tau_t}-1$, then one can change $\A_{t+1}$ by {$\phi_t(\A_{t+1})$, where $\phi_t(a)=\tau_t(a)_{\len{\tau_t}-1}$.
%
%$\phi_t(\A_{t+1})\subset{\A_{t}}$, so that $\card{\phi_t(\A_{t+1})}\le\card{\A_t}$.
%We also define, for $a\in\A_{t+2}$, $\tilde\tau_{t+1}(a)=\phi_t(\tau_{t+1}(a))\defeq(\phi_t(\tau_{t+1}(a)_i))_{0\le i<\len{\tau_{t+1}}}$.}
%Doing this inductively, we express the same limit set $\Omega_{\tilde\ttau}=\Omega_\ttau$ via a directive sequence of rank $\card{\A_0}$.
%
%Let us make an additional, similar, remark.
Most of our results assume non-erasingness and, in this subsection, injectivity.
%These are actually not restrictions because 
However one can inductively modify the alphabets so as to get these properties.
Indeed, if $\tau_0(a)$ is the empty word, then one can consider $\A_1'\defeq\A_1\setminus\{a\}$, over which the behavior of $\tau$ completely describes the configuration image set, and adapt $\tau_1$ (removing occurrences of $a$). Similarly, if $\tau_0(a)=\tau_0(b)$, then one can consider $\A_1'\defeq\A_1\setminus\{b\}$, and adapt $\tau_1$ (replacing occurrences of $b$ by $a$).
Doing this inductively for every $t\in{ \N}$, one gets a directive sequence with the same limit set and a non-greater rank (since alphabets are not increased in the process).
This may change ultimate periodicity or computability aspects, but this is not a problem for our results.
\end{remark}

\section{Comments}

Several of our results use the hypothesis of finite rank and one may wonder if ultimately this hypothesis is sufficient 
to have finite positive expansiveness. In the general context of zero-dimensional dynamical systems this is not the case, 
because any odometer has finite topological rank but is not finitely positively expansive. One possible question would be: which constraints, together with finite topological rank, imply finite positive expansiveness. \bigskip

%Let us notice that this is false for zero-dimensional dynamical systems in general: for instance an odometer has finite topological rank but is not finitely positively expansive. Under which constraint would finite rank 
%imply finite positive expansiveness in this more general context? 

%\begin{remark} Can we prove that it is false 
%for infinite rank Tœplitz shifts? (for instance having positive entropy). A priori results in this direction should be generalizable to two dimensions. 
%\end{remark}

%Another possibility is to search for characterizations under additional constraints: 

%\begin{question}
  %Can we characterize finitely positively expansive $\Z$-shift which are not the disjoint union of minimal shifts,
  %in particular, among transitive not minimal shifts?
%\end{question}

How complex can finitely positively expansive dynamical systems get? 
We already know that the entropy is zero when the function is a homeomorphism.
In particular, the entropy of finitely positively expansive $\Z$-shifts is always zero.
However for the examples considered, we observed that the 
complexity function is always bounded from above by a linear function $n \mapsto an+b$. Is it always the case? 
A possible direction towards this question would be to consider other topological invariants such as entropy dimension. \bigskip %(which can be done as well in the general context). %Another possibility is to restrict to substitutive shifts. In this restricted context, Pansiot's theorem provides a classification of equivalence classes of complexity functions. A strategy could be to search for examples of finitely positively expansive shifts in each of these classes, or prove that there can not be any. 

% It is also natural to wonder what happens for shifts in higher dimension. In this context one natural way to 
% generalize positive $n$-expansiveness in dimension $d$ would be to say that a $\Z^d$-shift $X$ is positively 
% $n$-expansive \textit{in direction $\theta$} if every admisssible half-infinite strip in direction $\theta$ admits at 
% most $n$ possible extensions. For more formal terminology about directional dynamics, one can refer to \cite{det}. It is known already that every $\Z^d$-shift which is positively expansive in one direction is finite (see for instance \cite{boylelind,det}). When $X$ is a shift of finite type (which is equivalent to having the shadowing property), its 
% subactions do not necessarily have the shadowing property as well, therefore Theorem~\ref{theorem.shadowing} does not apply to these subactions. In fact we may wonder if there exists some integer $n$ and an infinite $\Z^2$-shift of finite type which is positively $n$-expansive in every direction $\theta$.

\subsection{Higher-dimensional shifts}
Let us briefly sketch some comments and questions about higher-dimensional shifts.
For more formal terminology about directional dynamics, one can refer to \cite{det}.

% \begin{remark}\label{remark:dim2}
The subaction of a $\Z^d$-shift $X$ in some direction $\theta$ is positively $n$-expansive if every sufficiently wide half-infinite stripe in direction $\theta$ admits at most $n$ valid extensions (this definition remains relevant for irrational directions).
It is known, in this setting, that every $\Z^d$-shift which is expansive in every direction, or positively expansive in one direction, is finite (see for instance \cite{boylelind,det}).

Nevertheless, if $X$ is a shift of finite type (that is, has the $d$-dimensional shadowing property), its subactions need not have the shadowing property, so that Theorem~\ref{theorem.shadowing} does not apply.
This naturally leads to the following question:
%\begin{question}
  Does there exist $n\in\N$ and an infinite $\Z^2$-shift of finite type which is $n$-expansive in every direction, or positively $n$-expansive in some direction?
%Is there a natural condition to add to the shadowing property so that Theorem~\ref{theorem.shadowing} remains true for $\Z^d$-shifts, where $d \ge 2$? Is it possible to characterize the $\Z^d$-shifts of finite type which are finitely positively expansive?
%\end{question}
The corresponding question is already answered for sofic $\Z^2$-shifts, {as seen in the following example}, contradicting a naive generalization of Corollary~\ref{cor:sofic_exp}.
\begin{example}
  There exists an infinite sofic $\Z^2$-shift which is positively expansive in every direction, except $0$ and $\pi$, in which it is positively $2$-expansive.
  Indeed, consider any infinite effective positively $2$-expansive $\Z$-shift $X$ (for instance the Thue-Morse shift).
  Then the shift of all those configurations such that all lines are equal, and every line is in $X$ is sofic, thanks to \cite{effpsd,Aubrun2013}.
\end{example}

A natural related directional property is the following (note that each half-plane corresponds to a positive cone for some total group order of $\Z^2$, than can generalize what $\N$ is for $\Z$): a $\Z^d$-shift is weakly positively $n$-expansive in direction $\theta$ if every half-plane in direction $\theta$ admits at most $n$ valid extensions.
One can now ask,
%\begin{question}
  given a shift (of finite type), what are the constraints on possible maps $\theta\mapsto n$, where $n$ is the minimal integer such that the shift is positively $n$-expansive\resp{weakly} in direction $\theta$.
%\end{question}

The minimal substitutive subshifts from
%\begin{remark}
%  There exists an infinite minimal $\Z^2$-shift of finite type which is weakly positively $2$-expansive in every direction.
%  Indeed, the minimal variant of the Robinson shift (described in
\cite{minimalRobinson,ollinger}
%seems to be an example of infinite minimal $\Z^2$-shift of finite type which is weakly positively $8$-expansive in every direction. %satisfies these properties.{\red Can you explain a little bit more, why we have only finitely many extensions - is it a theorem in this paper? Which one?}
%\end{remark}
are good candidates to be looked into for these questions.

\section*{Acknowledgements}
P. Oprocha was supported by National Science Centre, Poland (NCN), grant no. 2019/35/B/ST1/02239.

%{\red Try construction of blowing up sequences starting from minimal set - PIOTR will put it here for further investigations - it should be in some paper of Downarowicz}

%Let us start with minimal shift $X$ and two points $x,y\in X$ such that $\sigma^n(x)\neq y$ for any $n$.
%Fix a new letter that is not in language of $X$, and call it $a$.
%Then produce a sequence of words $w_n$ in the following way. First $w_0=x_0 a^k$
%where $k\geq 0$ is the maximal number such that $x_{[0,k)}=y_{[0,k)}$. Note that
\bibliographystyle{alpha}
\bibliography{posnexp}

\end{document}